\documentclass{birkjour_t2}

\usepackage{amsfonts}
\usepackage{amsthm}
\usepackage{amsmath}
\usepackage{graphicx}
\usepackage{amssymb}
\usepackage{cases}
\usepackage{appendix}
\usepackage{fancyhdr}
\usepackage{subfig}
\usepackage{xcolor}
\usepackage{color}
\usepackage{graphicx}
\usepackage{lmodern}
\usepackage{multimedia}
\usepackage{lmodern}
\usepackage{tikz}
\usepackage{tikz,overpic}
\usepackage{multirow}
\usepackage{booktabs}
\usepackage{array}
\usepackage{empheq}
\usepackage{psfrag}
\usepackage{soul}
\usepackage{cite}
\usepackage{float}
\usepackage{epstopdf}
\newcommand{\cC}{{\mathcal C}}  
\newcommand{\cI}{{\mathcal I}}  
\newcommand{\cL}{{\mathcal L}}  
\newcommand{\cO}{{\mathcal O}}  

\newtheorem{theorem}{Theorem}[section]
\newtheorem{proposition}[theorem]{Proposition}
\newtheorem{lemma}[theorem]{Lemma}

\theoremstyle{definition}

\begin{document}

\title{\textbf{Geometry and numerical continuation of multiscale orbits in a nonconvex variational problem}}
\author[Annalisa Iuorio]{Annalisa Iuorio}
\address{%
Institute for Analysis and Scientific Computing, Vienna University of Technology\\
Wiedner Hauptstra{\ss}e 8-10,\\
1040, Vienna\\
Austria}
\email{annalisa.iuorio@tuwien.ac.at}
\author{Christian Kuehn}
\address{Faculty of Mathematics, Technical University of Munich\br
Boltzmannstra{\ss}e 3\br
85748 Garching bei M\"unchen,\br
Germany}
\email{ckuehn@ma.tum.de}

\author{Peter Szmolyan}
\address{%
Institute for Analysis and Scientific Computing, Vienna University of Technology\\
Wiedner Hauptstra{\ss}e 8-10,\\
1040, Vienna\\
Austria}
\email{szmolyan@tuwien.ac.at}




\date{January 27, 2017}

\begin{abstract}
We investigate a singularly perturbed, non-convex variational problem arising in materials science with a
combination of geometrical and numerical methods.
Our starting point is a work by Stefan M\"uller, where it is proven that the solutions of the variational 
problem are periodic and exhibit a complicated multi-scale structure. In order to get more insight into the
rich solution structure, we transform the corresponding Euler-Lagrange equation 
into a Hamiltonian system of first order ODEs and then use geometric singular perturbation theory to study 
its periodic solutions. Based on the geometric analysis we construct an initial periodic orbit to start numerical 
continuation of periodic orbits with respect to the key parameters. This allows us to observe the influence of the 
parameters on the behavior of the orbits and to study their interplay in the minimization process. Our results 
confirm previous analytical results such as the asymptotics of the period of minimizers predicted by M\"uller. 
Furthermore, we find several new structures in the entire space of admissible periodic orbits.
\end{abstract}

\maketitle

\section{Introduction}
\label{sec:intro}

The minimization problem we consider is to find
\begin{equation}
\label{eq:min}
\min_{u \in U} \left\{ \mathcal{I}^\varepsilon(u):=\int_0^1{\left( \varepsilon^2 u_{XX}^2 + 
W(u_X) + u^2 \right)} \, \mathrm{d}X \right\},
\end{equation}
where $U$ is a space containing all sufficiently regular functions $u:\left[0,1\right] 
\rightarrow \mathbb{R}$ of the spatial variable $X \in \left[0,1\right]$,
$0 < \varepsilon \ll 1$ is a small parameter, 
$u_X=\frac{\partial u}{\partial X}$, $u_{XX}=\frac{\partial^2 u}{\partial X^2}$, and the 
function $W$ is a symmetric, double well potential; in particular, here $W$ is chosen as
\begin{equation}
 \label{eq:W}
 W(u_X)=\frac{1}{4}(u_X^2-1)^2.
\end{equation}

This model arises in the context of coherent solid-solid phase transformation to describe 
the occurrence of simple laminate microstructures in one-space dimension. Simple laminates 
are defined as particular structures where two phases of the same material (e.g., 
austenite/martensite) simultaneously appear in an alternating pattern~\cite{Mu_rev}. This 
situation is shown schematically in Figure~\ref{fig:mu}(a). These and related structures 
have been intensively studied both in the context of geometrically linear 
elasticity~\cite{Kha83B, Kha69, Roi78} and in the one of fully nonlinear 
elasticity~\cite{Abe01,Abe96B,Bal04,Bal87,Bha03B,Dol03B,Ped00B,Pit03B,Tru95,Tru96}. A comparison 
between these two approaches is given by Bhattacharya~\cite{Bha93}. We focus here on the 
one-dimensional case starting from the work of M\"uller~\cite{Mu}, but a 2D approach has also 
been proposed~\cite{Koh94, Giu12, Hea}. 
An alternative choice of the functional $W$ which sensibly simplifies energy calculations
for equilibria has been recently adopted by Yip~\cite{Yip}. The same functional with more general boundary
conditions has been treated by Vainchtein et al~\cite{Vain1}.
In all these cases, very significant theoretical and experimental advances have been reached.
Nevertheless, many interesting features concerning the asymptotics and dynamics of these problems can still be explored.

We start from the one-dimensional model~\eqref{eq:min}-\eqref{eq:W} analyzed 
by M\"uller and introduce a different approach based on geometric singular perturbation 
theory~\cite{GSPT,Kue} which allows us to better understand the critical points of the 
functional $\mathcal{I}^\varepsilon$ and to obtain an alternative method eventually able 
to handle more general functionals.

      \begin{figure}[!ht]
      \centering
      \psfrag{(a)}{(a)}
      \psfrag{(b)}{(b)}
      \psfrag{eps}{\footnotesize{$\mathcal{O}(\varepsilon)$}}
      \psfrag{eps13}{\footnotesize{$\mathcal{O}(\varepsilon^\alpha)$}}
      \psfrag{A}{\textcolor{white}{\footnotesize{A}}}
      \psfrag{M}{\footnotesize{M}}
      \psfrag{u'}{\footnotesize{$u_X$}}
      \psfrag{x}{\footnotesize{$X$}}
      \psfrag{1}{\footnotesize{$1$}}
      \psfrag{-1}{\footnotesize{$-1$}}
      \includegraphics[scale=0.6]{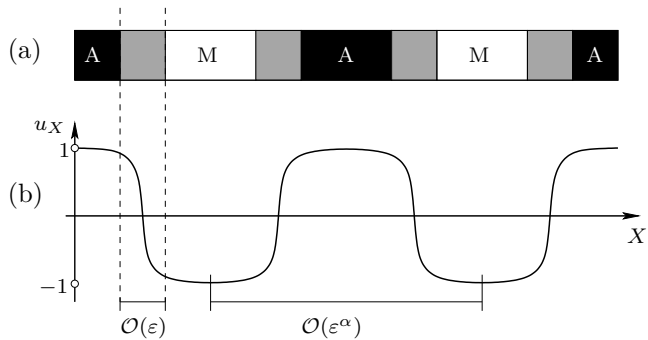}
      \caption{Schematic representation of simple laminates microstructures 
			as periodic solutions. (a) Microstructures in one space dimension: austenite 
			(A) and martensite (M) alternate, while the transition area is shown in gray. 
			(b) Structure in space of the variable $u_X$, whose values $\pm 1$ represent 
			the two different phases of the material of width of order $\cO(\varepsilon^\alpha)$, 
			with $\alpha = 1/3$ for minimizers (as shown in~\cite{Mu}) and $\alpha=0$ for other critical points.
			The width of the transition interval is of order $\cO(\varepsilon)$. }\label{fig:mu}
      \end{figure} 

In~\cite{Mu} minimizers are proven to exhibit a periodic multi-scale structure (Figure~\ref{fig:mu}): 
a fast scale of order $\cO(\varepsilon)$ describes the ``jumps'' between the two values of the 
derivative $u_X$, and a slow scale of order $\cO(\varepsilon^{1/3})$ represents the distance between 
two points with equal value of $u_X$. From a physical viewpoint, the two values of the derivative 
$u_X=\pm 1$ model the two different phases of the material. The jumps describe the transition 
between the phases and the regions with almost constant values of $u_X$ correspond to parts of 
the material occupied by the same phase.
One of the key results in~\cite{Mu} consists in an
asymptotic formula for the period of minimizing solutions, when the solution space $U$ is chosen
as the set of all ${u \in H^2(0,1)}$ subject to Dirichlet boundary conditions. 
For $\varepsilon \rightarrow 0$, the period $P^{\varepsilon}$ behaves as
\begin{equation}
 \label{eq:P_mue}
  P^{\varepsilon}=2(6 A_0 \varepsilon)^{1/3} + \mathcal{O}(\varepsilon^{2/3}),
\end{equation}
where $A_0=2 \int_{-1}^1 \! W^{1/2}(w) \, \mathrm{d}w$. 

The approach based on fast-slow analysis of the Euler-Lagrange equation applied here
allows us to identify geometrically certain classes
of periodic orbits. These orbits are used as starting solutions for numerical continuation using the software 
package~\texttt{AUTO}~\cite{AU}. This powerful tool has been adopted for example by
Grinfeld and Lord~\cite{GrLo} in their numerical analysis of small amplitude periodic solutions of \eqref{eq:min}.
We provide here a detailed study of periodic solutions based upon one-parameter continuation in the parameters 
$\varepsilon$ and $\mu$. It turns out that several fold bifurcations of periodic orbits  
structure the parameter space. A numerical comparison with the law \eqref{eq:P_mue} will be presented, 
by means of a minimization process of the functional $\mathcal{I}^{\varepsilon}$ along certain families 
of periodic orbits. Our work also leads to new insights into the dependence of the period 
on the parameters $\varepsilon$ and $\mu$ for non-minimizing sequences of periodic orbits.\medskip

The paper is structured as follows: Section~\ref{sec:EL} introduces the approach based on geometric singular 
perturbation theory using the intrinsic multi-scale structure of the problem. We describe the 
transformation of the Euler-Lagrange equation associated to the functional $\mathcal{I}^\varepsilon$ 
into a multiscale ODE system, along with the decomposition of periodic orbits into slow and fast pieces 
using the Hamiltonian function. 
We identify a family of large amplitude singular periodic orbits and prove their persistence for $\epsilon$ small.
A crucial point is the construction of an initial periodic 
orbit for $\varepsilon \neq 0$ in order to start numerical continuation: the strategy we use is 
illustrated in Section~\ref{sec:numerics}, where the continuation of the orbits with respect to the main 
parameters is also performed. This section includes also the comparison between the analytical expression 
of the period given by M\"uller and our numerical results as well as the general parameter study of 
periodic solutions. Section~\ref{sec:conclusion} is devoted to conclusions and an outline for future 
work.

\section{The Euler-Lagrange equation as a fast-slow system}
\label{sec:EL}

In this section, the critical points (not only the minimizers) of the functional 
$\mathcal{I}^{\varepsilon}$ are analyzed. A necessary condition they 
have to satisfy is the Euler-Lagrange equation~\cite{Dac15}. The Euler-Lagrange
equation associated to $\mathcal{I}^{\varepsilon}$ is the singularly perturbed, 
fourth order equation
\begin{equation}
\label{eq:EL}
\varepsilon^2u_{XXXX}-\frac{1}{2}\sigma(u_X)_X+u=0,
\end{equation}
where $\sigma(u_X)=W'(u_X)=u_X^3-u_X$. Equation~\eqref{eq:EL} can be rewritten 
via 
\[
 \begin{aligned}
  w &:= u_X, \\
  v &:= -\varepsilon^2 w_{XX}+\frac12 \sigma(w), \\
  z &:= \varepsilon w_X,
 \end{aligned}
\]
as an equivalent system of first order ODEs
\begin{equation}
\label{eq:fusyS}
\begin{aligned}
\dot{u} &=w,\\
\dot{v} &=u,\\
\varepsilon \dot{w} &= z,\\
\varepsilon \dot{z} &= \frac{1}{2} (w^3-w)-v,
\end{aligned}
\end{equation}
where $\frac{\mathrm{d}}{\mathrm{d}X}=\dot{}$\ .
Equations~\eqref{eq:fusyS} exhibit the structure of a (2,2)-fast-slow system, 
with $u, v$ as slow variables and $w, z$ as fast variables. We recall that a 
system is called \emph{(m,n)-fast-slow}~\cite{Feni,GSPT,Kue} when it has the form
\begin{equation}
\label{eq:sfs}
\begin{aligned}
\varepsilon \dot{x} &= f(x,y,\varepsilon),\\
\dot{y} &= g(x,y,\varepsilon),
\end{aligned}
\end{equation}
where $x \in \mathbb{R}^m$ are the \emph{fast} variables and $y \in \mathbb{R}^n$ 
are the \emph{slow} variables. The re-formulation of system~\eqref{eq:sfs}
on the fast scale is obtained by using the change of variable $\xi=\frac{X}{\varepsilon}$, i.e.
\begin{equation}
\label{eq:sff}
\begin{aligned}
\frac{\mathrm{d} x}{\mathrm{d} \xi} = x' &= f(x,y,\varepsilon),\\
\frac{\mathrm{d} y}{\mathrm{d} \xi} = y' &= \varepsilon g(x,y,\varepsilon).
\end{aligned}
\end{equation}
On the fast scale, system~\eqref{eq:fusyS} has the form
\begin{equation}
\label{eq:fusyF}
\begin{aligned}
u' &=\varepsilon w,\\
v' &=\varepsilon u,\\
w' &= z,\\
z' &= \frac{1}{2} (w^3-w)-v,
\end{aligned}
\end{equation}
which for $\varepsilon>0$ is equivalent to \eqref{eq:fusyS}.

The system possesses the unique equilibrium 
\begin{equation}
 p_0= \left( 0,0,0,0 \right),
\end{equation}
which is a center, since the eigenvalues of the Jacobian are all purely imaginary. 
An important property of system~\eqref{eq:fusyS} is stated in the following result:

\begin{lemma}
\label{lem:Ham}
Equations \eqref{eq:fusyS} and \eqref{eq:fusyF} are singularly perturbed Hamiltonian 
systems
{\arraycolsep=0.7pt\def\arraystretch{1.4}
\begin{equation}
\label{spHs}
\begin{array}{rcr}
    \dot{u} &=&- \frac{\partial H}{\partial v}\\
    \dot{v} &=& \frac{\partial H}{\partial u}\\
    \varepsilon \dot{w} &=& -\frac{\partial H}{\partial z}\\
    \varepsilon \dot{z} &=& \frac{\partial H}{\partial w},
\end{array}
\qquad 
\overset{\textstyle{\xi=\frac{X}{\varepsilon}}}{\Leftrightarrow}
\qquad 
\begin{array}{rcr}
  u' &=&-\varepsilon \frac{\partial H}{\partial v}\\
  v' &=&\varepsilon \frac{\partial H}{\partial u}\\
  w' &=& -\frac{\partial H}{\partial z}\\
  z' &=& \frac{\partial H}{\partial w},
\end{array}
\end{equation}
}
i.e., they are Hamiltonian systems with respect to the symplectic form 
$\mathrm{d}z \land \mathrm{d} w + \frac{1}{\varepsilon} \mathrm{d}v \land 
\mathrm{d}u$ and with Hamiltonian function 
\begin{equation}
\label{Ham}
H(u,v,w,z)=\frac{1}{8}(4 u^2 - 8 v w - 2 w^2 + w^4 - 4 z^2).
\end{equation}
\end{lemma}

\begin{proof}
The result follows by differentiating~\eqref{Ham} with respect to 
the four variables. For more background on fast-slow Hamiltonian systems 
of this form see~\cite{Gel02}.
\end{proof}

Since the Hamiltonian~\eqref{Ham} is a first integral of the system, the dynamics 
take place on level sets, defined by fixing $H(u,v,w,z)$ to a constant value 
$\mu\in\mathbb{R}$. This allows us to reduce the dimension of the system by one, which 
we use both in analytical and numerical considerations.\medskip  

One main advantage in the use of geometric singular perturbation theory is that the 
original problem can be split into two subsystems by analyzing the singular limit 
$\varepsilon \rightarrow 0$ on the slow scale~\eqref{eq:sfs} and on the fast 
scale~\eqref{eq:sff}. The subsystems are usually easier to handle. Under suitable
conditions the combination of both subsystems allows us to obtain information for 
the full system when $0 < \varepsilon \ll 1$. In particular, if one can construct 
a singular periodic orbit by combining pieces of slow and fast orbits, then the existence of 
a periodic orbit $\mathcal{O}(\varepsilon)$-close to the singular one for small 
$\varepsilon \neq 0$ can frequently be proven under suitable technical 
conditions by tools from geometric singular perturbation 
theory~\cite{Feni,GSPT,Kue,ST}.\medskip

The slow singular parts of an orbit are derived from the \emph{reduced problem} (or 
\emph{slow subsystem}), obtained by letting $\varepsilon \rightarrow 0$ in~\eqref{eq:sfs}
\begin{equation}
\label{eq:rp0}
\begin{aligned}
0 &= f(x,y,0),\\
\dot{y} &= g(x,y,0),
\end{aligned}
\end{equation}
which describes the slow dynamics on the \emph{critical manifold} 
\begin{equation}
\cC_0:=\{(x,y)\in\mathbb{R}^m\times \mathbb{R}^n:f(x,y,0)=0.\} 
\end{equation}
Considering $\varepsilon \rightarrow 0$ on 
the fast scale \eqref{eq:sff} yields the \emph{layer problem} (or 
\emph{fast subsystem})
\begin{equation}
\label{eq:lp0}
\begin{aligned}
x' &= f(x,y,0),\\
y' &= 0,
\end{aligned}
\end{equation}
where the fast dynamics is studied on ``layers'' with constant values of the 
slow variables. Note that $\cC_0$ can also be viewed as consisting of equilibrium
points for the layer problem. $\cC_0$ is called \emph{normally hyperbolic} if 
the eigenvalues of the matrix $\textnormal{D}f_x(p,0)\in\mathbb{R}^{m\times m}$ 
do not have zero real parts for $p\in\cC_0$. For normally hyperbolic invariant
manifolds, Fenichel's Theorem applies and yields the existence of a \emph{slow
manifold} $\cC_\varepsilon$. The slow manifold lies at a distance $\cO(\varepsilon)$
from $\cC_0$ and the dynamics on $\cC_\varepsilon$ is well-approximated by the reduced
problem; for the detailed technical statements of Fenichel's Theorem we refer 
to~\cite{Feni,GSPT,Kue}.\medskip

In our Hamiltonian fast-slow context, we focus on the analysis of families of 
periodic orbits for system~\eqref{eq:fusyS} which are parametrized by the 
level set parameter~$\mu$. The first goal is to geometrically construct  
periodic orbits in the singular limit $\varepsilon=0$. The reduced 
problem is given by
\begin{equation}
\label{eq:RPO}
\begin{aligned}
\dot{u} &= w,\\
\dot{v} &= u,
\end{aligned}
\end{equation}
on the critical manifold (see Fig.~\ref{fig:CrMan})
\begin{equation}
\label{CrMan}
\mathcal{C}_0=\left\{ (u, v, w, z) \in \mathbb{R}^4\ :\ z=0,\ v=\frac{1}{2}(w^3-w) \right\}.
\end{equation}
The equations of the layer problem are 
\begin{equation}
\label{eq:LP}
\begin{aligned}
w' &= z\\
z' &= \frac{1}{2}(w^3-w)-\bar{v},
\end{aligned}
\end{equation}
on ``layers'' where the slow variables are constant ($u=\bar{u},\ v=\bar{v}$).
Note that for Hamiltonian fast-slow systems such as \eqref{spHs}, both reduced and layer
problems are Hamiltonian systems with one degree of freedom.

\subsection{The Reduced Problem}
\label{ssec:reduced}

Equations \eqref{eq:RPO} describe the reduced problem on $\cC_0$, if $w$ is considered
as a function of $(u, v)$ on $\cC_0$.

\begin{lemma}
\label{lem:C0}
$\mathcal{C}_0$ is normally hyperbolic except for two fold lines 
\begin{equation}
\begin{aligned}
\cL_-& =\left\{\left(u,\left(w_-^3-w_-\right)/2 ,w_-,0\right)\in\mathbb{R}^4\right\},\\
\cL_+& =\left\{\left(u,\left(w_+^3-w_+\right)/2 ,w_+,0\right)\in\mathbb{R}^4\right\},
\end{aligned}
\end{equation}
where $w_\pm$ are defined by $\sigma'(w_\pm)=0$, i.e., $w_\pm=\pm1/\sqrt{3}$. 
For $p\in \cL_\pm$, the matrix $\mathrm{D}_x f(p,0)$ has a double zero 
eigenvalue.
\end{lemma}
      
The lines $\cL_\pm$ naturally divide $\mathcal{C}_0$ into three parts
\[
\mathcal{C}_{0,l}=\mathcal{C}_0 \cap \left\{ w < w_- \right\},\ 
\mathcal{C}_{0,m}=\mathcal{C}_0 \cap \left\{w_- \leq w \leq w_+\right\},\ 
\mathcal{C}_{0,r}=\mathcal{C}_0 \cap \left\{ w > w_+ \right\},
\]
as shown in Figure~\ref{fig:CrMan}. The submanifolds involved in our analysis are 
only $\mathcal{C}_{0,l}$ and $\mathcal{C}_{0,r}$, which are normally hyperbolic. More
precisely, $\mathcal{C}_{0,l}$ and $\mathcal{C}_{0,r}$ are of \emph{saddle-type}, since 
the matrix $\mathrm{D}_x f(p,0)$ along them always has two real eigenvalues of opposite
sign. We remark that saddle-type critical manifolds have played an important role in the
history of fast-slow systems in the context of the travelling wave problem for the 
FitzHugh-Nagumo equation, see for 
example~\cite{JonesKopellLanger,KrupaSandstedeSzmolyan,GuckenheimerKuehn1}.

      \begin{figure}[!ht]
      \psfrag{C0}{$\mathcal{C}_0$}
      \psfrag{Clm}{$\mathcal{C}_l^\mu$}
      \psfrag{Crm}{$\mathcal{C}_r^\mu$}
      \psfrag{w}{$w$}
      \psfrag{v}{$v$}
      \psfrag{u}{$u$}
      \centering
      \includegraphics[scale=0.4]{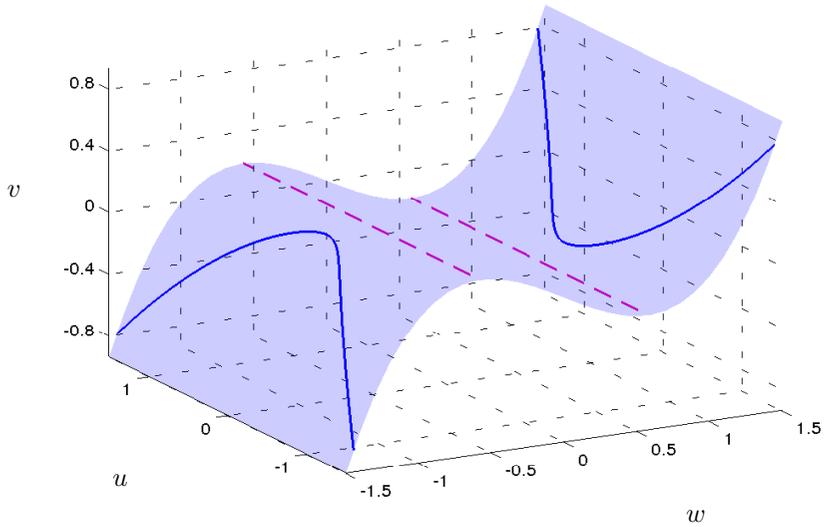}
      \caption{Critical manifold $\cC_0$ in $(w,u,v)$-space. 
			The magenta dashed lines are the fold lines $\cL_\pm$. The blue solid curves correspond 
			to $\mathcal{C}_{0,l}^\mu$ and $\mathcal{C}_{0,r}^\mu$, i.e., the intersection 
			of $\mathcal{C}_{0,l}$ and $\mathcal{C}_{0,r}$ and the hypersurface $H(u,v,w,z)=\mu$ 
			for $\mu=0$. }\label{fig:CrMan}
      \end{figure} 
      
\begin{lemma}
On $\cC_0-\cL_\pm$, the flow of the reduced system is, up to a time rescaling, given by
\begin{equation}
\label{slsy}
\begin{aligned}
\dot{u} &= (3w^2-1)w,\\
\dot{w} &= 2u.
\end{aligned}
\end{equation}
\end{lemma}

\begin{proof}
We differentiate $v=\frac12(w^3-w)$ with respect to $X$, re-write the equation in 
$(u,w)$-variables and apply the time rescaling corresponding to the multiplication
of the vector field by the factor $(3w^2-1)$ (cf.~\cite[Sec.7.7]{Kue}). On 
$\mathcal{C}_{0,m}$ this procedure changes the direction of the flow, but it does not affect
the parts of the critical manifold involved in our analysis.
\end{proof}

The Hamiltonian function allows us to restrict our attention to two subsets of 
$\mathcal{C}_{0,l}^\mu$ and $\mathcal{C}_{0,r}^\mu$ by fixing the value of $\mu$. 
Analyzing the slow flow on these two normally hyperbolic branches, we see that
$u$ decreases along $\cC_{0,l}$ and increases along $\cC_{0,r}$ as shown in Figure~\ref{fig:CrManMu}. 

      \begin{figure}[!ht]
      \psfrag{Clm}{$\mathcal{C}_l^\mu$}
      \psfrag{Crm}{$\mathcal{C}_r^\mu$}
      \psfrag{w}{\footnotesize{$w$}}
      \psfrag{u}{\footnotesize{$u$}}
      \centering
      \includegraphics[scale=0.5]{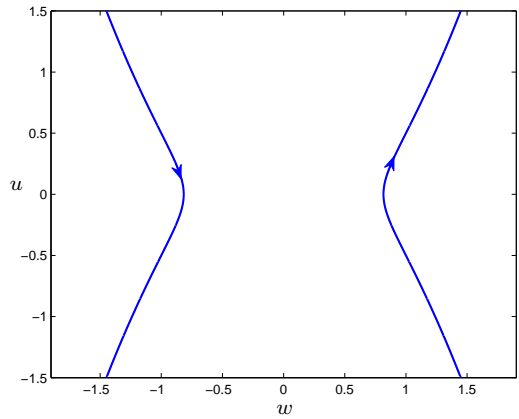}
      \caption{$\mathcal{C}_{0,l}^\mu$ and $\mathcal{C}_{0,r}^\mu$ in $(w,u)$-space
      with $\mu=0$; cf.~Figure~\ref{fig:CrMan}.}\label{fig:CrManMu}
      \end{figure}

\subsection{The Layer Problem}
\label{ssec:layer}

The layer problem is obtained by setting $\varepsilon=0$ in~\eqref{eq:fusyF}. We 
obtain a two-dimensional Hamiltonian vector field on ``layers'' where the slow variables 
are constant ($u=\bar{u},v=\bar{v}$)
\begin{equation}
\label{eq:fasy}
\begin{aligned}
w' &= z,\\
z' &= \frac{1}{2}(w^3-w)-\bar{v}.
\end{aligned}
\end{equation}
The two branches $\mathcal{C}_{0,l}^\mu$ and $\mathcal{C}_{0,r}^\mu$ are 
hyperbolic saddle equilibria for the system~\eqref{eq:fasy} for every value 
of $\bar{u}, \bar{v}$. To construct a singular limit periodic orbit we are 
particularly interested in connecting orbits between equilibria of the 
layer problem.

      \begin{figure}[!ht]
      \psfrag{w}{\footnotesize{$w$}}
      \psfrag{z}{\footnotesize{$z$}}
      \centering
      \includegraphics[scale=0.5]{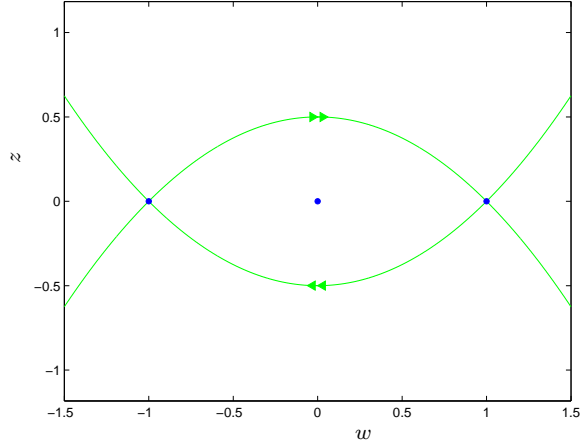}
      \caption{Fast flow in the $(w,z)$-space for~\eqref{eq:fasy}. Equilibria are 
      marked with blue dots and the stable and unstable manifold trajectories in 
      green. The heteroclinic fast connections are indicated with double arrows.}
      \label{fig:FaFl}
      \end{figure}
      
\begin{lemma}
\label{lem:het}
The layer problem~\eqref{eq:fasy} has a double heteroclinic connection if and only 
if $\bar{v}=0$. These are the only possible heteroclinic connections of the layer 
problem~\eqref{eq:fasy}.
\end{lemma}

\begin{proof}
System~\eqref{eq:fasy} is Hamiltonian, with $\bar{v}$ as a parameter and Hamiltonian
function
 \[
  H_f(w,z)=-\frac{z^2}{2} + \frac{w^4}{8} - \frac{w^2}{4} - \bar{v} w.
 \]
The lemma follows easily by discussing the level curves of the Hamiltonian;
for the convenience of the reader we outline the argument.

Indexing the level set value of $H_f$ as $\theta$, the solutions 
of~\eqref{eq:fasy} are level curves $\{H_f(w,z)=\theta\}$. The equilibria 
of~\eqref{eq:fasy} are $\{z=0,w=w_l,w_m,w_r\}$; here $w_l,w_m,w_r$ are the 
three solutions of 
\begin{equation}
\label{eq:3sol}
2\bar{v}-w^3+w=0
\end{equation}
which depend upon $\bar{v}$. We only have to
consider the case where there are at least two real equilibria $w_l$ and $w_r$ 
which occurs for $\bar{v}\in[-1/(3\sqrt{3}),1/(3\sqrt{3})]$. Let 
\[
H_f(w_l,0)=:\theta_l,\qquad H_f(w_r,0)=:\theta_r  
\]
and note that since~\eqref{eq:3sol} is cubic we can calculate $\theta_{l,r}$
explicitly. To get a heteroclinic connection we must have $\theta_l=\theta_r$
and by an explicit calculation this yields the condition $\bar{v}=0$. Hence, 
heteroclinic connections of~\eqref{eq:fasy} can occur only if $\bar{v}=0$. 
For $\bar{v}=0$ one easily finds that the relevant equilibria are located at
$w_l=-1$ and $w_r=1$ so that $\theta_l=-1/8=\theta_r$. The double heteroclinic
connection is then explicitly given by the curves $\{z=\pm\frac12(1-w^2)\}$ 
(see also Figure~\ref{fig:FaFl}).
\end{proof}
  
The next step is to check where the relevant equilibria of the layer problem 
are located on the critical manifold $\cC_0^\mu$ for a fixed value of the 
parameter $\mu$ since we have a level set constraint for the full system. Using 
Lemma~\ref{lem:het} one must require $w=\pm1,v=0$ while $z=0$ is the critical
manifold constraint, hence
\[
H(u,0,\pm1,0)=\frac12 u^2-\frac18\stackrel{!}{=}\mu.
\]
Therefore, the transition points where fast jumps from $\cC_{0,l}$ to $\cC_{0,r}$
and from $\cC_{0,r}$ to $\cC_{0,l}$ are possible are given by
\begin{equation}
\label{eq:jpoints}
\cC_0^\mu\cap \{v=0,w=\pm 1\}=\left\{u=\pm \sqrt{2\mu+\frac14},v=0,w=\pm1,z=0\right\}. 
\end{equation}
Observe that fast orbits corresponding to positive values of $u$ connect 
$\mathcal{C}_{0,r}^\mu$ to $\mathcal{C}_{0,l}^\mu$, while the symmetric orbits 
with respect to the $u=0$ plane connect 
$\mathcal{C}_{0,l}^\mu$ to $\mathcal{C}_{0,r}^\mu$.\medskip

Recall that $w = \pm 1$ represent the two phases of the material.
Hence, the heteroclinic orbits of the layer problem can be interpreted as
instantaneous transitions between these phases.


\subsection{Singular Fast-Slow Periodic Orbits}
\label{ssec:po}

The next step is to define singular periodic orbits by combining pieces of orbits of
the reduced and layer problem. Figure~\ref{fig:SiOr} illustrates the situation. The entire singular orbit 
$\gamma_0^\mu$ is obtained connecting two pieces of orbits of the reduced problem with
heteroclinic orbits of the fast subsystem for a 
fixed value of $\mu$, see Figure~\ref{fig:SiOr}(a). The configuration of the two-dimensional
critical manifold and the singular periodic orbit is indicated in Figure~\ref{fig:SiOr}(b).\medskip

      \begin{figure}[!ht]
      \begin{minipage}{.5\textwidth}
      \psfrag{w}{\footnotesize{$w$}}
      \psfrag{z}{\footnotesize{$z$}}
      \psfrag{u}{\footnotesize{$u$}}
      \centering
      \includegraphics[scale=0.5]{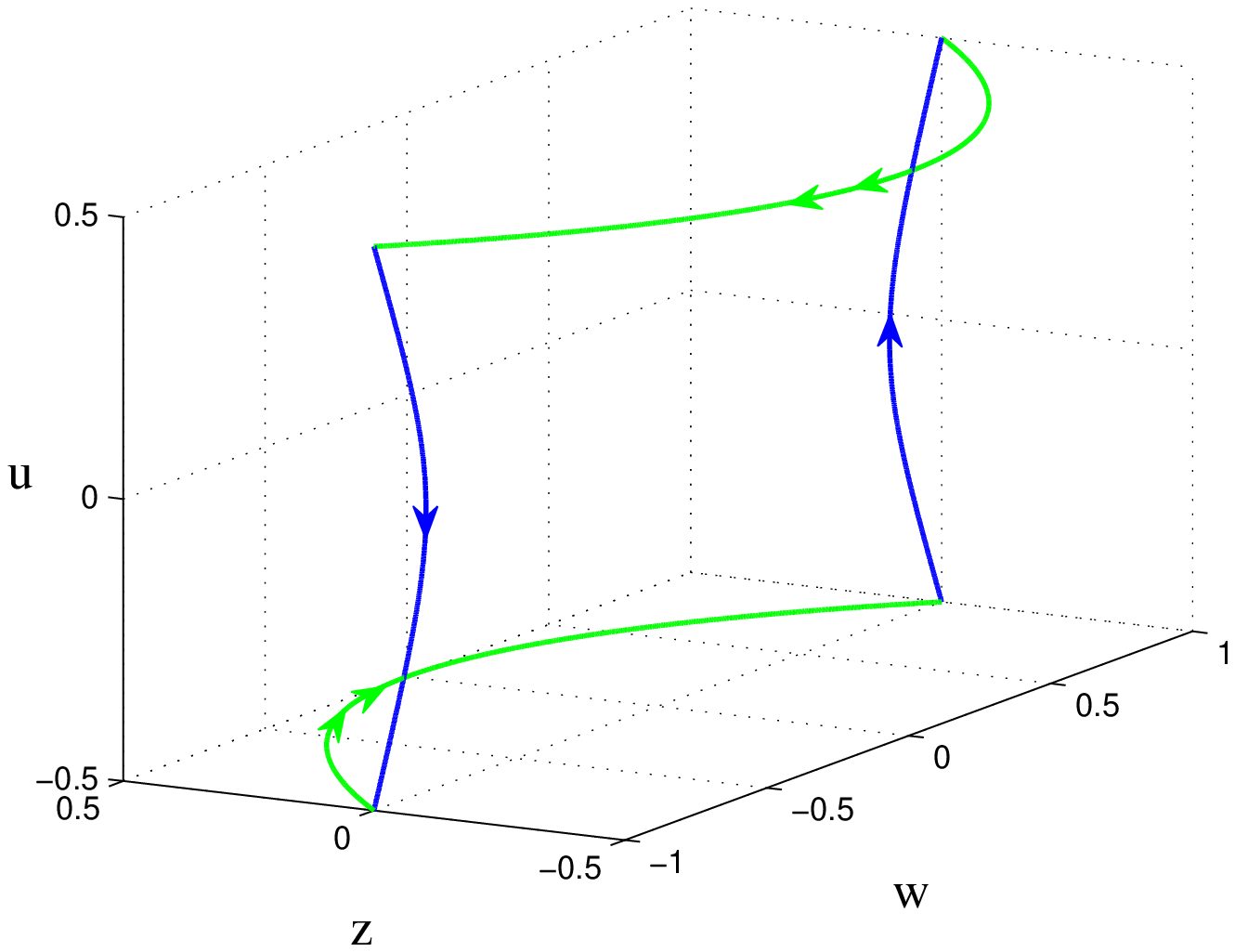}
      \end{minipage}%
      \begin{minipage}{.5\textwidth}
        \psfrag{(b)}{(b)}
        \psfrag{w}{\footnotesize{$w$}}
        \psfrag{u}{\footnotesize{$u$}}
        \psfrag{v}{\footnotesize{$v$}}
	\centering
	\includegraphics[scale=0.31]{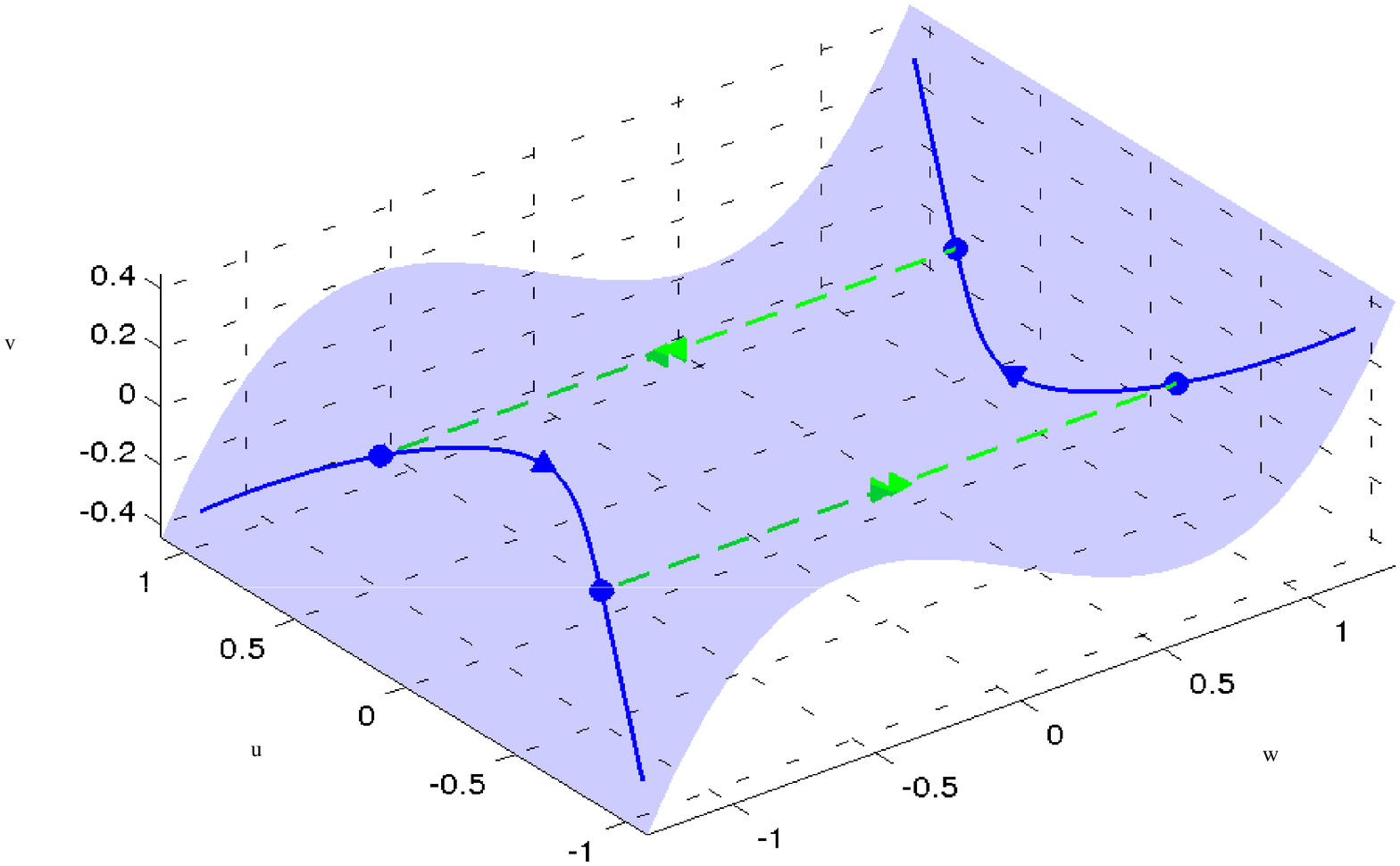}
      \end{minipage}
      \caption{Singular periodic orbit $\gamma_0^\mu$ for a fixed value of $\mu$ ($\mu=0$), 
      obtained by composition of slow (blue) and fast (green) pieces. (a) Orbit in 
      $(w,z,u)$-space. (b) Orbit in the $(w,u,v)$-space. The fast pieces are indicated via dashed lines
      to illustrate the fact we are here considering their projection in $(w,u,v)$, 
      while they actually occur in the $(w,z)$-plane. Consequently, they do not intersect $\cC_{0,m}$.}
      \label{fig:SiOr}      
      \end{figure}

Here we are only interested in singular periodic orbits which have nontrivial slow
\emph{and} fast segments. Therefore, we do need transitions points from the fast 
subsystem to the slow subsystem. This requirement implies, by using the 
result~\eqref{eq:jpoints}, the lower bound $\mu>-\frac18$. A second requirement
we impose is that the slow subsystem orbits lie inside the normally hyperbolic 
parts $C_{0,l}^\mu$ and $C_{0,r}^\mu$. The $u$-coordinate of the slow segment
closest to the lines $\cL_\pm$ is located at $u=0$. Hence, we calculate 
the value of the Hamiltonian under the condition that the slow trajectory is tangent to
$\cL_\pm$ which yields
\begin{equation}
H\left(0,\frac12(w_\pm^3-w_\pm),w_\pm,0\right)=\frac{1}{24}. 
\end{equation}
Combining these considerations with the results from 
Sections~\ref{ssec:reduced}-\ref{ssec:layer} gives the following result on the existence 
of singular periodic orbits ($\varepsilon=0$):

\begin{proposition}
\label{prop:porbit}
For $\varepsilon=0$, the fast-slow system~\eqref{eq:fusyS},\eqref{eq:fusyF} 
has a family of periodic orbits $\{\gamma_0^\mu\}_{\mu}$ consisting 
of precisely two fast and two slow subsystem trajectories with slow parts
lying entirely in $\cC_{0,l}$ and $\cC_{0,r}$ if and only if 
\begin{equation}
\label{murange}
\mu \in I_{\mu}, \hspace{1.5cm} I_{\mu} := 
\left(-\frac{1}{8},\frac{1}{24}\right). 
\end{equation}
\end{proposition}


The persistence of these periodic orbits for $0 < \varepsilon \ll 1$ on each individual surface level
of the Hamiltonian can be proven by using an argument based on the theorem introduced
by Soto-Trevi{\~n}o in~\cite{ST}.
\begin{theorem}
\label{thm:ST}
For every $\mu \in I_\mu$ and for $\varepsilon > 0$ sufficiently small,
there exists a locally unique periodic orbit of the fast-slow system~\eqref{eq:fusyS},\eqref{eq:fusyF} that is
$\mathcal{O}(\varepsilon)$ close to the corresponding singular \mbox{orbit $\gamma_0^\mu$}.
\end{theorem}
\begin{proof}
The Hamiltonian structure of the system suggests to study the individual levels as parametrized
families by directly applying the Hamiltonian function as a first integral
to reduce the dimension of the system. At first sight, a convenient choice is to express $v$ 
as a function of the variables $(u, w, z)$ and $\mu$
\begin{equation}
\label{v}
v=\frac{4u^2-8\mu-2w^2+w^4-4z^2}{8w}.
\end{equation}
Consequently, equations \eqref{eq:fusyS} transform into a $(2,1)$-fast-slow system
\begin{equation}
\label{eq:resyS}
\begin{aligned}
\dot{u} &= w,\\
\varepsilon\dot{w} &= z,\\
\varepsilon\dot{z} &= \frac{1}{2} (w^3-w)-\frac{4u^2-8\mu-2w^2+w^4-4z^2}{8w}.\\
\end{aligned}
\end{equation}
Theorem $1$ in \cite{ST} for a $C^r$ $(r \geq 1)$ $(2,1)$-fast-slow system guarantees the persistence of periodic orbits
consisting of two slow pieces connected by heteroclinic orbits for $0 < \varepsilon \ll 1$ when the following conditions hold:
\begin{itemize}
 \item The critical manifolds are one-dimensional and normally hyperbolic (given by Lemma~\ref{lem:C0}).
 \item The intersection between $W^u(\cC_{0,l})$ (resp. $W^u(\cC_{0,r})$) and $W^s(\cC_{0,r})$ (resp. $W^s(\cC_{0,l})$) is transversal (confirmed by Lemma~\ref{lem:het}, see Fig.~\ref{fig:trint}).
 \item The full system possesses a singular periodic orbit and the slow flow on the critical manifolds is transverse to touch-down and take-off sets,
	which reduce to 0-dimensional objects in this case as we explicitly obtained in \eqref{eq:jpoints}.
\end{itemize}

      \begin{figure}[!ht]
      \psfrag{C0}{\footnotesize{$\mathcal{C}_0$}}
      \psfrag{Wlm}{\tiny{$W^u(\mathcal{C}_0^l)$}}
      \psfrag{Wrm}{\tiny{$W^s(\mathcal{C}_0^r)$}}
      \psfrag{w}{\footnotesize{$w$}}
      \psfrag{z}{\footnotesize{$z$}}
      \psfrag{v}{\footnotesize{$v$}}
      \centering
      \includegraphics[scale=0.35]{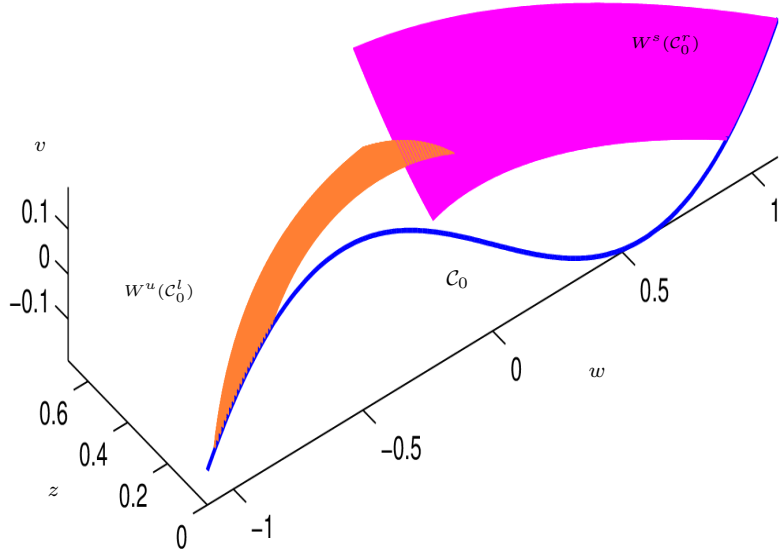}
      \caption{Transversal intersection in the $(w,z,v)$ space between $W_u(\cC_{0,l})$ (in orange) and $W_s(\cC_{0,r})$ (in magenta). The blue line represents the critical manifold $\cC_0$.}
      \label{fig:trint}
      \end{figure}
      
However, system~\eqref{eq:resyS} appears to be nonsmooth at $w=0$ and the fast orbits  necessarily cross $w=0$. To overcome this (apparent) difficulty we use other charts 
for the manifold $H(u,v,w,z) = \mu$ for parts of the singular orbit close to $w=0$. Instead of~\eqref{v} we now express $u$ as a function of the other variables, i.e.
\begin{equation}
\label{u}
u = \pm \frac{1}{2} \sqrt{8 v w + 2 w^2 - w^4 + 4 z^2 + 8 \mu}.
\end{equation}
This leads to the following description of the dynamics:
\begin{itemize}
 \item System~\eqref{eq:resyS} describes the dynamics on the slow pieces away from $w=0$;
 \item The heteroclinic connection corresponding to $u>0$ is expressed by
      \begin{equation}
      \label{eq:resy2S}
      \begin{aligned}
	v' &= + \frac{\varepsilon}{2} \sqrt{8 v w + 2 w^2 - w^4 + 4 z^2 + 8 \mu},\\
	w' &= z,\\
	z' &= \frac{1}{2} (w^3-w)-v.\\
      \end{aligned}
      \end{equation}
 \item The heteroclinic connection corresponding to $u<0$ is expressed by
      \begin{equation}
      \label{eq:resy2S}
      \begin{aligned}
	v' &= - \frac{\varepsilon}{2} \sqrt{8 v w + 2 w^2 - w^4 + 4 z^2 + 8 \mu},\\
	w' &= z,\\
	z' &= \frac{1}{2} (w^3-w)-v.\\
      \end{aligned}
      \end{equation}
\end{itemize}
If we consider system \eqref{eq:fusyS} as a smooth dynamical system on the manifold defined by $H=\mu$, the proof given in \cite{ST} (based on proving the transversal intersection of two manifolds obtained by flowing suitably chosen initial conditions forward and backward in time) goes through without being affected by the fact that we have to work with several coordinate systems, as described above.
\end{proof}


System~\eqref{eq:resyS} has two parameters $\mu,\varepsilon$, which naturally
leads to the question how periodic orbits deform and bifurcate when the two 
parameters are varied. Furthermore, the fast-slow structure with orbits consisting
of two fast jumps and two slow segments as shown in Figure~\ref{fig:SiOr} and the 
three-dimensional form~\eqref{eq:resyS} provide analogies to the travelling wave
frame system obtained from the partial differential equation version of the 
FitzHugh-Nagumo~\cite{FitzHugh,NAG} (FHN) equation. The three-dimensional fast-slow 
FHN system has been studied in great detail using various fast-slow systems 
techniques~(see e.g.~\cite{CarterSandstede,GuckenheimerKuehn1,JonesKopellLanger,
KrupaSandstedeSzmolyan}).
One particular approach to investigate the FHN parameter space efficiently is to 
employ numerical continuation methods~\cite{GuckenheimerKuehn3,Sneydetal}. In fact,
numerical approaches to FHN have frequently provided interesting conjectures and 
thereby paved the way for further analytical studies. Adopting this approach, we are 
going to investigate the problem~\eqref{eq:resyS} considered here using numerical 
continuation to gain better insight into the structure of periodic orbits.

\section{Numerical Continuation}
\label{sec:numerics}

This section is devoted to the numerical investigation of the the critical points 
of the functional $\mathcal{I}^{\varepsilon}$ via the Euler-Lagrange 
formulation~\eqref{eq:resyS}. A powerful tool for such computations 
is~\texttt{AUTO}~\cite{AU}. \texttt{AUTO} is able to numerically
track periodic orbits depending upon parameters using a combination of a boundary
value problem (BVP) solver with a numerical continuation algorithm. Using
such a framework for fast-slow systems often yields a wide variety
of interesting numerical and visualization results; for a few recent examples we refer 
to~\cite{DesrochesKaperKrupa,DesrochesKrauskopfOsinga2,SMST,
GuckenheimerLaMar,GuckenheimerMeerkamp,Tsaneva-AtanasovaOsingaRiessSherman}.\medskip 

The first task one has to deal with is the construction of a starting orbit for fixed 
$\varepsilon \neq 0$. For~\eqref{eq:resyS} this is actually a less trivial task than
for the FHN equation as we are going to explain in Section~\ref{ssec:startorb}. 
In Section~\ref{ssec:startorb}, we are also going to construct a starting periodic 
orbit based upon the geometric insights of Section~\ref{sec:EL}.\medskip

Once the starting periodic orbit is constructed, we use \texttt{AUTO} to perform 
numerical continuation in both parameters $\mu$ and $\varepsilon$. This yields 
bifurcation diagrams and the solutions corresponding to some interesting points 
on the bifurcation branches. Then, the connection between the parameters in the 
minimization process is investigated, in order to numerically determine the 
correspondence that leads to the functional minimum. Finally, a comparison with 
the period law~\eqref{eq:P_mue} predicted by M\"uller~\cite{Mu} is performed. 

\subsection{Construction of the starting orbit}
\label{ssec:startorb}

As indicated already, the construction of a starting periodic orbit is not trivial:

\begin{itemize}
\item{The singular orbit itself, obtained by matching slow and fast subsystem orbits 
for a fixed value of $\mu$ and for $\varepsilon=0$, cannot be used owing to re-scaling 
problems (the fast pieces would all correspond to $x=0$).}
\item{The computation of a full periodic orbit using a direct initial value solver 
approach for \mbox{$0<\varepsilon\ll 1$} is hard to perform since the slow manifolds are 
of saddle type and an orbit computed numerically would diverge from them exponentially 
fast~\cite{SMST}.}
\item{Matching slow segments obtained with a saddle-type algorithm~\cite{SMST,Kristiansen} 
and fast parts computed with an initial value solver may cause problems at the points 
where the four pieces should match.}
\item{In contrast to the FitzHugh-Nagumo case~\cite{GuckenheimerKuehn3,Sneydetal}, the periodic 
orbits we are looking for cannot be detected as Hopf bifurcations from the zero equilibrium. 
In that case, we could use \texttt{AUTO} to locate such bifurcations and then find a periodic 
orbit for $0 < \varepsilon \ll 1$ fixed by branch-switching at the Hopf bifurcation point. In 
our case, however, the origin $p_0$ is a center equilibrium, and an infinite number of periodic 
orbits exist around it in the formulation~\eqref{eq:fusyS}.} 
\item{Starting continuation close to the equilibrium $p_0$ is difficult due to its degenerate nature ($w=0$).
} 
\end{itemize}

Our strategy is to use the geometric insight from Section~\ref{sec:EL} in combination with
a slow manifolds of saddle-type (SMST) algorithm and a homotopy approach. We construct an 
approximate starting periodic orbit using a value of $\mu$ which leads to a short 
``time'' spent on the slow parts of the orbits, so that the saddle-type branches 
do not lead to numerical complications. Then we use an SMST algorithm to find a suitable 
pair of starting points lying extremely close to the left and right parts of the slow manifolds 
$\cC^\mu_{\varepsilon,l}$ and $\cC^\mu_{\varepsilon,r}$. In the last step we employ 
numerical continuation to study the values of $\mu$ we are actually interested in; this 
is the homotopy step.\medskip 

The value $\mu=-\frac18$ is peculiar, since in this case the singular slow segments for~\eqref{eq:resyS} reduce 
to two points
\begin{equation}
\mathcal{C}_{0,l}^{-1/8} = \left\{ (0,-1,0) \right\},\qquad \text{and} \qquad
\mathcal{C}_{0,r}^{-1/8} = \left\{ (0,1,0) \right\},
\end{equation}
i.e., touch-down and take-off sets for the fast dynamics coincide in this case.
The range~\eqref{murange} we are considering does not include $\mu=-\frac18$; 
however, this property makes it an excellent candidate for the first step of our strategy. 
Indeed, we know already from the geometric analysis in Section~\ref{sec:EL} that the time
spent near slow manifolds is expected to be very short in this case.

Although it is still not possible to compute the full orbit using forward/backward integration, 
we can compute two halves, provided we choose the correct initial condition. We aim to find 
a point on the slow manifolds $\cC^\mu_{\varepsilon,l}$ and $\cC^\mu_{\varepsilon,r}$ as an
initial value. The SMST algorithm~\cite{SMST} helps to solve this problem. The 
procedure is based on a BVP method to compute slow manifolds of saddle-type in fast-slow 
systems. Fixing $\varepsilon$ and $\mu$, we select manifolds $B_l$ and $B_r$, which are 
transverse to the stable and unstable eigenspaces of $\mathcal{C}_{0,l}^\mu$ and 
$\mathcal{C}_{0,r}^\mu$, respectively (Figure \ref{fig:SMST}). The plane $B_l$ and the line
$B_r$ provide the boundary conditions for the SMST algorithm.\medskip
 
      \begin{figure}[H]
      \psfrag{Clm}{$\mathcal{C}_l^\mu$}
      \psfrag{Bl}{$B_l$}
      \psfrag{Br}{$B_r$}
      \centering
      \includegraphics[scale=0.65]{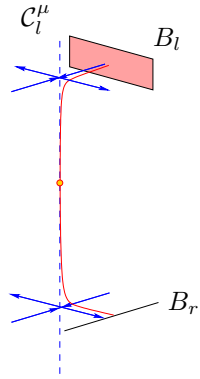}
      \caption{Schematic representation of the SMST algorithm applied to $\mathcal{C}_{0,l}^\mu$ 
      (an analogous situation occurs for $\mathcal{C}_{0,r}^\mu$). The critical manifold is 
      indicated by a dotted blue line, while the red line represents the slow manifold for 
      $\varepsilon=0.001$. The orange point corresponds to $(0,w_L,0)$, which 
      actually belongs to both manifolds.}\label{fig:SMST}
      \end{figure}

Implementing the algorithm for $\mu=-\frac{1}{8}$ and $\varepsilon=0.001$ for~\eqref{eq:resyS}
shows that there are actually two points $(0,w_L,0)$ and $(0,w_R,0)$ which 
are contained in the slow manifold even for $\varepsilon \neq 0$ as well as in the critical 
manifold $\cC_0$. From the geometric analysis in Section~\ref{sec:EL} we know that at 
$\mu=-\frac18$ the take-off and touch-down points coincide and a singular 
double-heteroclinic loop exists for $v=0$. This motivates the choice of $(u,v,w,z)=(0,0,w_L,0)$ 
and $(u,v,w,z)=(0,0,w_R,0)$ in the following algorithm: a numerical integration 
of the full four-dimensional problem~\eqref{eq:fusyS} forward and backward in $x$ is performed, 
imposing the Hamiltonian constraint using a projective algorithm. The computation is stopped 
once the hyperplane $\{w=0\}$ is reached. The full periodic orbit is then constructed by matching 
two symmetric pieces together.\medskip

In principle, there are different ways how one may arrive at a useful construction of a
highly accurate starting periodic orbit. In our context, the geometric analysis guided the
way to identify the simplest numerical procedure, which is an approach that is likely to be
successful for many other non-trivial fast-slow numerical continuation problems.

\subsection{Continuation in $\mu$}
\label{ssec:contmu}

A detailed analysis of the critical points' dependence on the Hamiltonian is performed. 
The value of $\mu$ can be arbitrarily chosen only in the interval $I_{\mu}$, while 
$\varepsilon$ is fixed to $0.001$. Continuation is performed on system~\eqref{eq:resyS} 
using the initial orbit obtained numerically in Section~\ref{ssec:startorb}. Starting at
$\mu=-\frac{1}{8}$, \texttt{AUTO} is able to compute the variation of the orbits up to 
$\mu=\frac{1}{24}$. The bifurcation diagram of the period $P$ with respect to the 
parameter $\mu$ is shown in Figure~\ref{fig:Fig1}(a).\medskip 

      \begin{figure}[!ht]
      \begin{minipage}{.5\textwidth}
        \psfrag{(a)}{(a)}
	\psfrag{mu}{\footnotesize{$\mu$}}
	\psfrag{P}{\footnotesize{$P$}}
	\centering
	\includegraphics[scale=0.55]{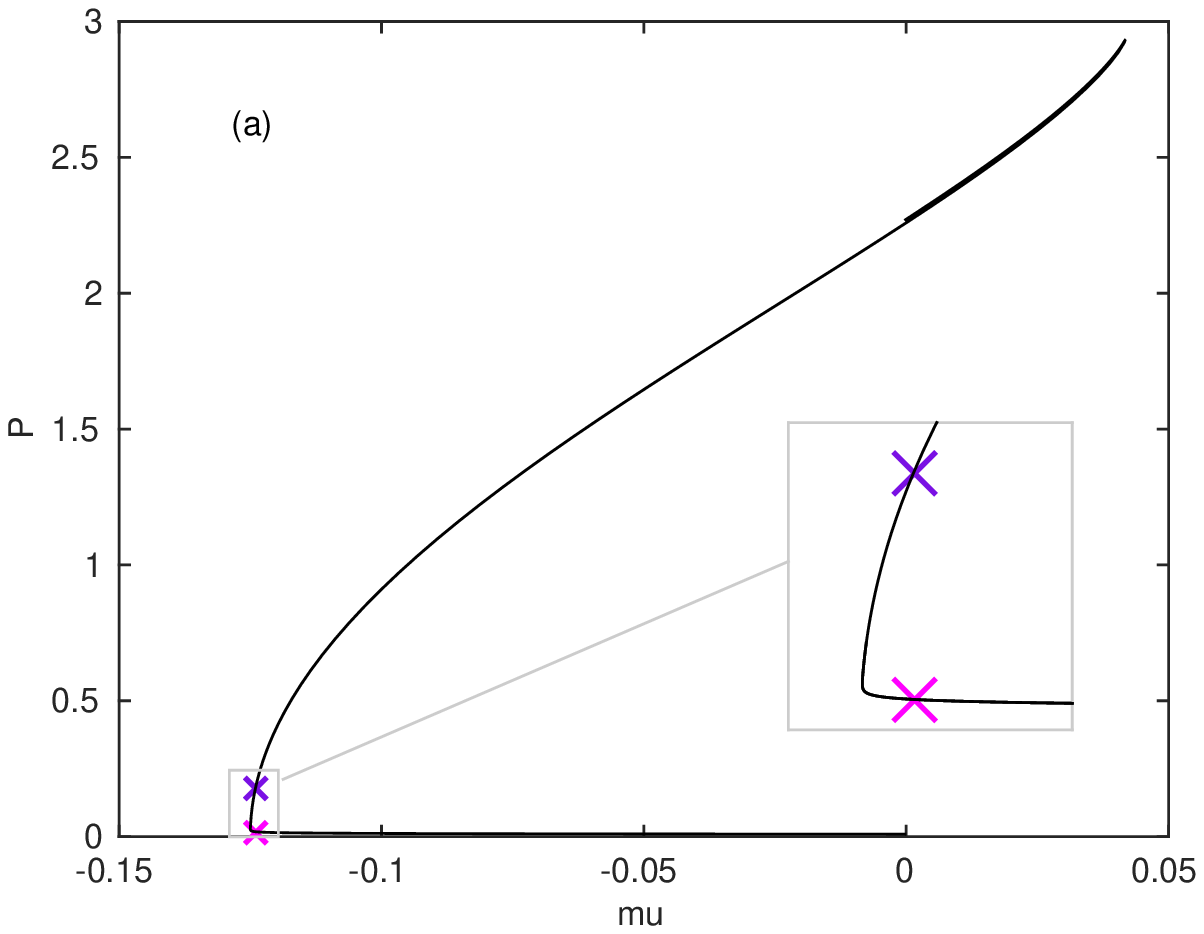}
      \end{minipage}%
      \begin{minipage}{.5\textwidth}
        \psfrag{(b)}{(b)}
        \psfrag{w}{\footnotesize{$w$}}
        \psfrag{u}{\footnotesize{$u$}}
        \psfrag{z}{\footnotesize{$z$}}
	\centering
	\includegraphics[scale=0.55]{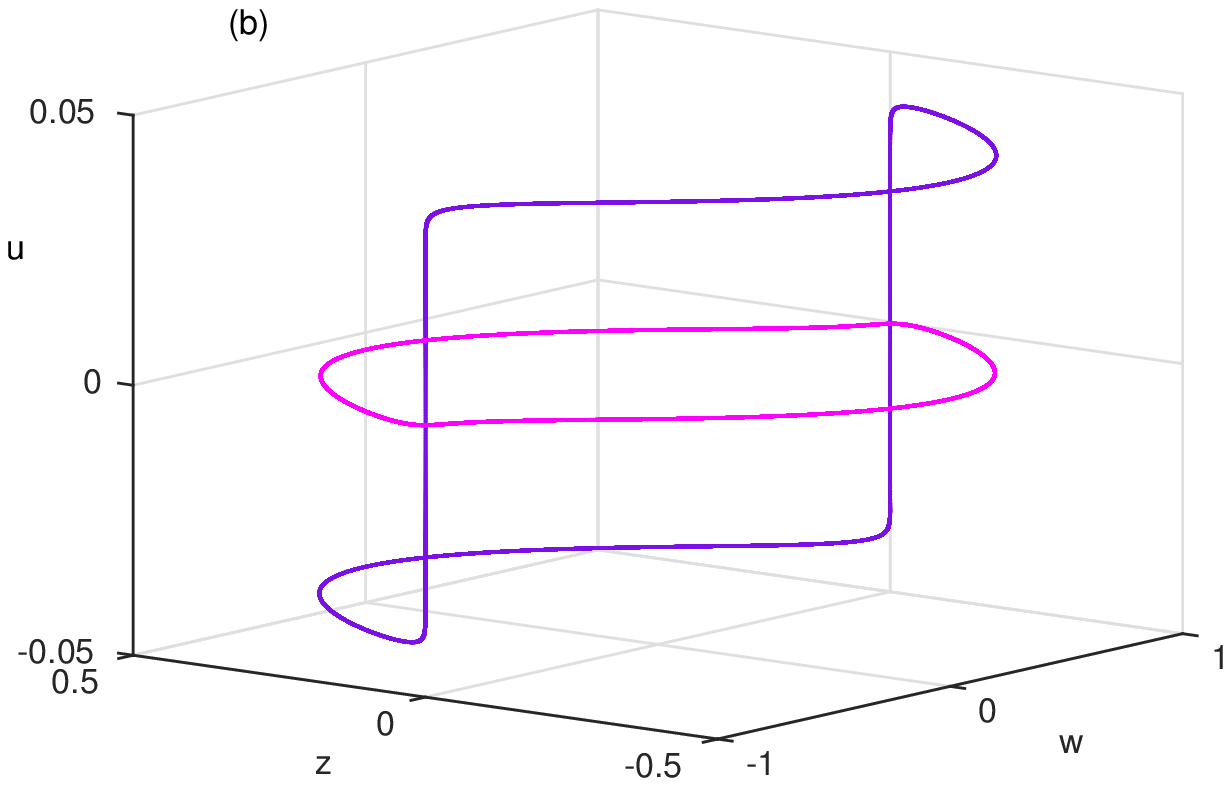}
      \end{minipage}
      \caption{Continuation in $\mu$: (a) bifurcation diagram in $(\mu, P)$-space, 
      where two periodic solutions corresponding to $\mu=-0.124$ are marked by crosses; 
      (b) corresponding solutions in $\left(w,z,u \right)$-space: the one on the lower 
      branch (magenta) is almost purely fast, while the one on the upper branch (purple) 
      contains long non-vanishing slow pieces.}
      \label{fig:Fig1}      
      \end{figure}
      
The first/upper branch of the continuation displays fast-slow orbits corresponding perturbations 
of the singular ones $\left\{ \gamma_0^{\mu} \right\}_{\mu \in I_{\mu}}$ for fixed 
$\varepsilon \neq 0$. As predicted by the geometric analysis we observe that decreasing 
$\mu$ reduces the length of the slow parts, so that the orbits almost correspond to the double
heteroclinic one analytically constructed at $\mu = -\frac{1}{8}$; see Figure~\ref{fig:Fig1}. Near 
$\mu=-\frac{1}{8}$ the bifurcation branch has a fold in $(\mu,P)$-space leading to 
the second/lower bifurcation branch. The difference between the orbits on the two 
branches for a fixed value of $\mu$ is shown in Figure~\ref{fig:Fig1}(b). Along the 
second branch, periodic solutions around the center equilibrium appear, which collapse 
into it with increasing $\mu$ (Figure \ref{fig:Fig2}(b)).\medskip

Furthermore, numerical continuation robustly indicates that the upper branch has another 
fold when continued from $\mu=0$ to higher values of $\mu$ as shown in 
Figure~\ref{fig:Fig2}(a). The orbits obtained by fixing a value of $\mu$ on the upper branch
and its continuation after the fold differ only because of the appearance of two new fast 
parts near the plane $\{u=0\}$ as shown in Figure~\ref{fig:Fig2}(a). We conjecture that 
these parts arise due to the loss of normal hyperbolicity at $\cL_\pm$; see also 
Section~\ref{sec:conclusion}.

      \begin{figure}[H]
      \begin{minipage}{.38\textwidth}
        \psfrag{(a)}{(a)}
	\psfrag{mu}{\footnotesize{$\mu$}}
	\psfrag{P}{\footnotesize{$P$}}
	\centering
	\includegraphics[scale=0.55]{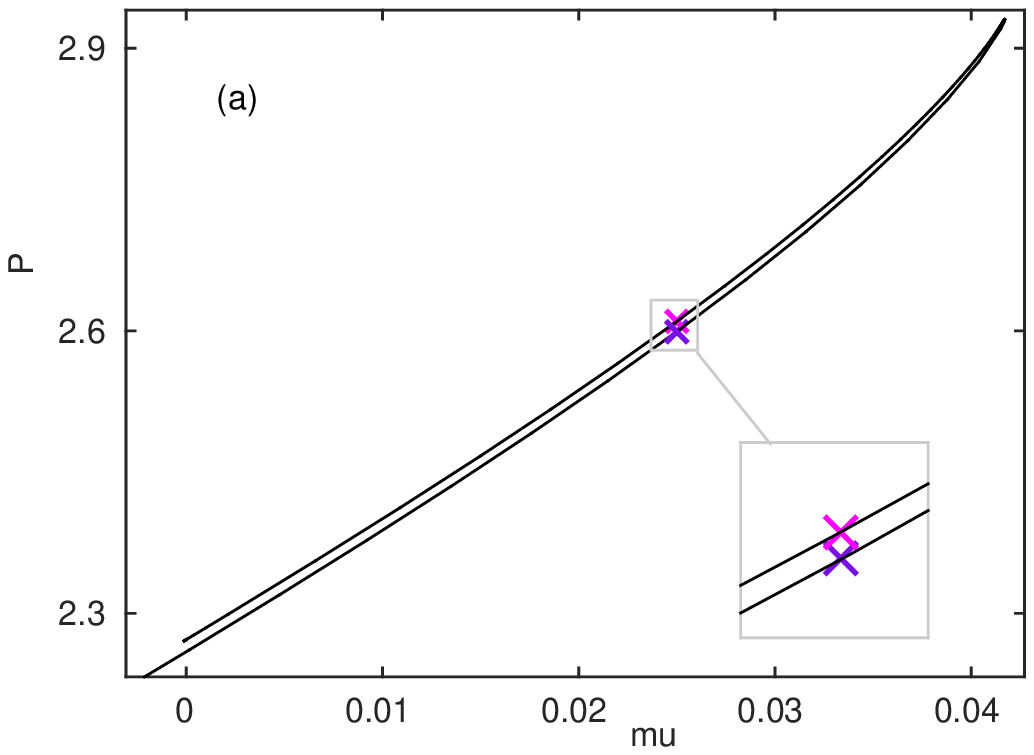}
      \end{minipage}
      \begin{minipage}{.3\textwidth}
        \psfrag{w}{\footnotesize{$w$}}
        \psfrag{u}{\footnotesize{$u$}}
        \psfrag{z}{\footnotesize{$z$}}
        \psfrag{a1}{\footnotesize{(a1)}}
	\centering
	\includegraphics[scale=0.45]{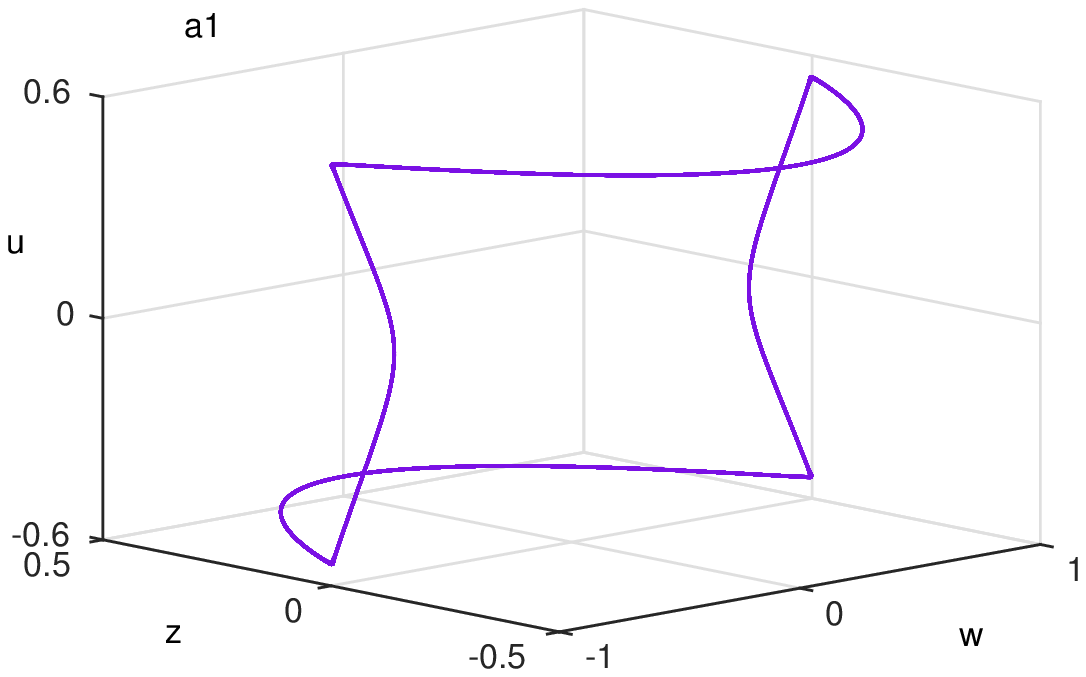}
      \end{minipage}
      \begin{minipage}{.3\textwidth}
        \psfrag{w}{\footnotesize{$w$}}
        \psfrag{u}{\footnotesize{$u$}}
        \psfrag{a2}{\footnotesize{(a2)}}
	\centering
	\includegraphics[scale=0.45]{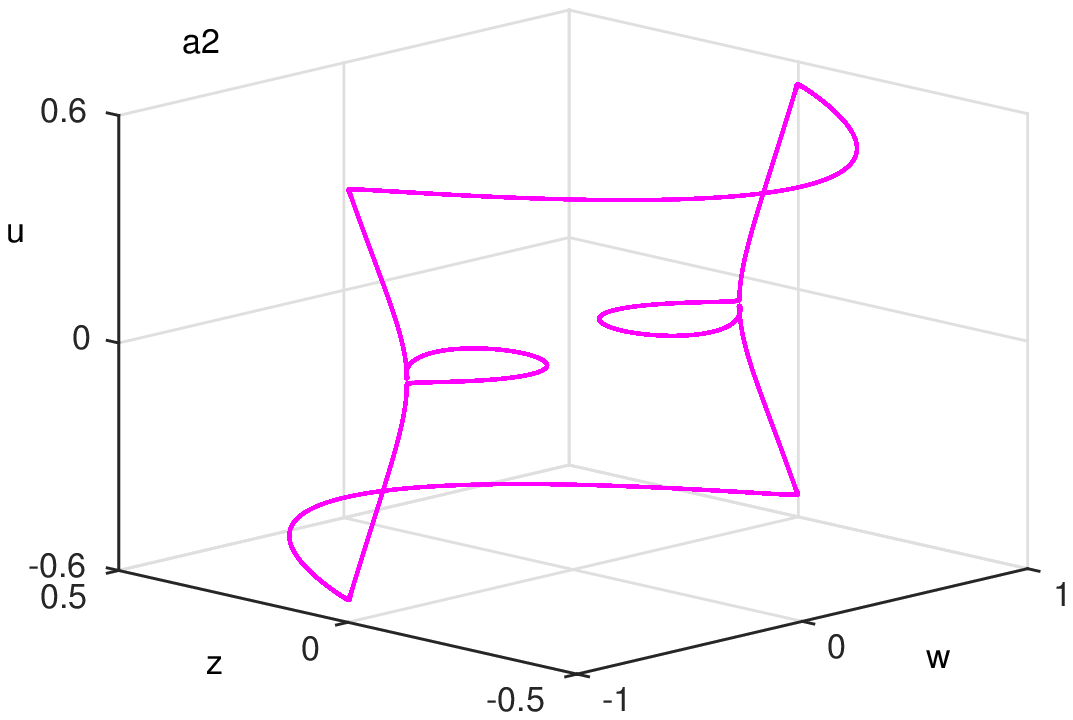}
      \end{minipage}      
      \\
      \hspace{-1cm}
      \begin{minipage}{.5\textwidth}
        \psfrag{(b)}{(b)}
	\psfrag{mu}{\footnotesize{$\mu$}}
	\psfrag{P}{\footnotesize{$P$}}
	\centering
	\includegraphics[scale=0.6]{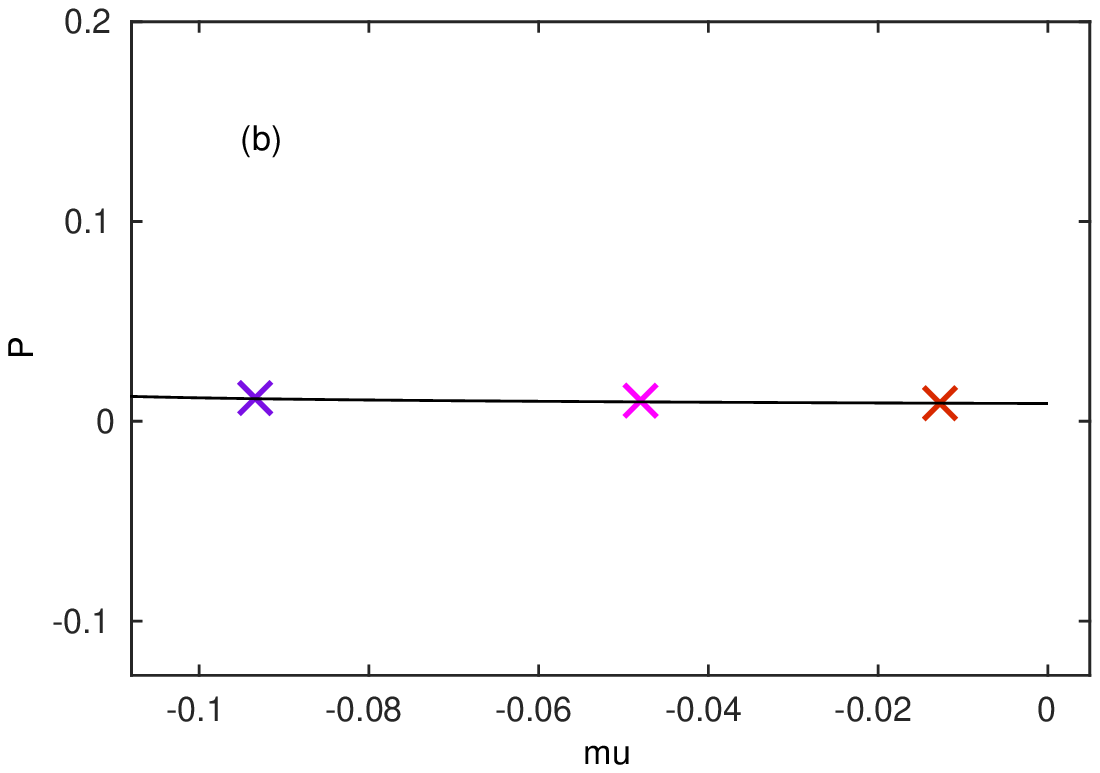}
      \end{minipage}
      \begin{minipage}{.5\textwidth}
        \psfrag{(b)}{(b)}
        \psfrag{w}{\footnotesize{$w$}}
        \psfrag{u}{\footnotesize{$u$}}
        \psfrag{z}{\footnotesize{$z$}}
        \psfrag{b1}{\footnotesize{(b1)}}
        \psfrag{x10}{\scalebox{.5}{$\times 10^{-3}$}}
	\centering
	\includegraphics[scale=0.42]{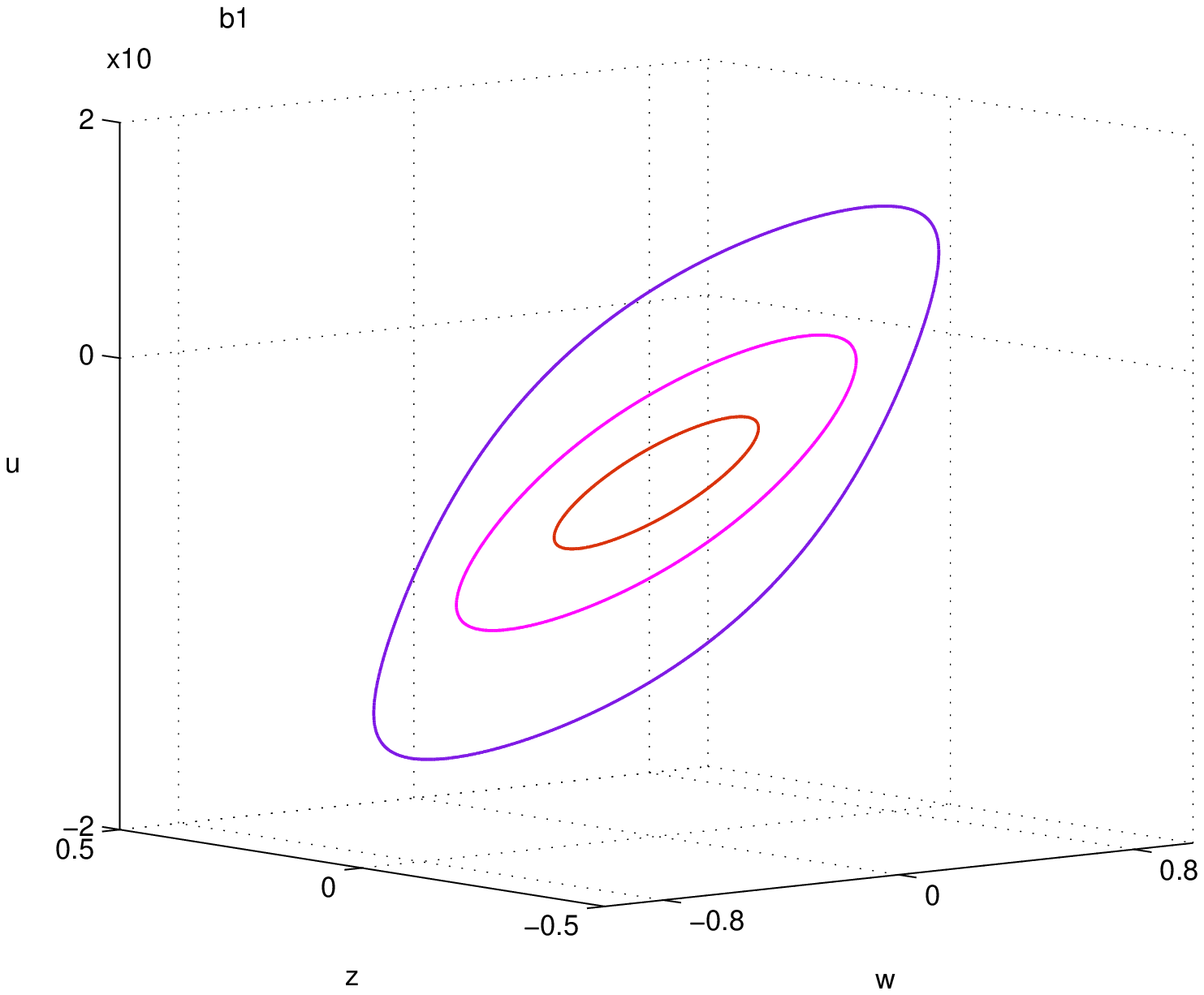}
      \end{minipage}      
      \caption{Continuation in $\mu$. (a) Zoom on the upper part of the 
      bifurcation diagram in-$(\mu,P)$ space, where two periodic orbits 
      corresponding to $\mu=0.0025$ are marked by crosses. (a1)-(a2) The orbits are 
      shown in $\left(w,z,u\right)$-space. 
      The periodic orbit on the bottom part of the upper branch (purple) 
      corresponds to analytical expectations with two fast and two slow segments. 
      The periodic orbit on the top part of the upper branch (magenta) includes 
      two new fast ``homoclinic excursions''. (b) Zoom on the lower part of the 
      bifurcation diagram in $(\mu, P)$-space, where three solutions are marked. 
      (b1) The solutions in phase space all correspond to 
      periodic orbits around the center equilibrium $p_0$; note that the scale 
      in the $u$-coordinate is extremely small so the three periodic orbits almost 
      lie in the hyperplane $\{u=0\}$.}
      \label{fig:Fig2}      
      \end{figure}

\subsection{Continuation in $\varepsilon$}
\label{ssec:conteps}

We perform numerical continuation in $\varepsilon$ by fixing three values of 
$\mu$ in order to capture the behavior of the solutions for the range $I_{\mu}$ 
from Proposition~\ref{prop:porbit}. We consider $\mu_l \approx -\frac{1}{8}$ with
$\mu_l>-\frac18$, $\mu_c \approx 0$, and $\mu_r \approx \frac{1}{24}$ with 
$\mu_r<\frac{1}{24}$; or more precisely $\mu_l = -0.12489619925$, 
$\mu_c = 1.5378905702\cdot 10^{-5}$, and $\mu_r = 0.04100005066$. 
For each of these values, we find two bifurcation branches connected via a fold
in $(\mu,P)$-space; see Figure~\ref{fig:Fig4}.\medskip

The bifurcation diagrams and the associated solutions shown in Figure~\ref{fig:Fig4}
nicely illustrate the dependence of the period on the singular perturbation parameter
$\varepsilon$. When $\varepsilon\rightarrow 0$ there are two very distinct limits for 
the period $P=P(\varepsilon)$ (Figure~\ref{fig:Fig4}(a)-(b)) depending whether we are 
on the upper and lower parts of the main branch of solutions. In the case with 
$\mu_r\approx \frac{1}{24}$ when orbits come close to non-hyperbolic singularities on
$\cC_0$, we actually seem to observe that $P(0)$ 
seems to be independent on whether we consider the upper or lower part of the branch
(see Figure~\ref{fig:Fig4}(c)). Furthermore, functional forms of $P(\varepsilon)$ are 
clearly different for small $\varepsilon$ so the natural conjecture is that there 
is no universal periodic scaling law if we drop the functional minimization constraint.\medskip

The deformation under variation of $\varepsilon$ of the periodic orbits in $(w,z,u)$-space 
is also interesting. For $\mu=\mu_l$ (Figure~\ref{fig:Fig4}(a)), we observe that the upper 
branch corresponds to the solutions that we expect analytically from 
Proposition~\ref{prop:porbit} consisting of two fast and two slow segments when 
approaching $\varepsilon=0$. A similar scenario occurs also 
for the other values of $\mu$ (Figure~\ref{fig:Fig4}(b) and Figure~\ref{fig:Fig4}(c)). 
When the $\varepsilon$ value is too large, or when we are on 
a different part of the branch of solutions, the orbits closed to the equilibrium of 
the full system or additional pieces resembling new fast contributions appear.

      \begin{figure}[!ht]
      \hspace{-1cm}
      \begin{minipage}{.5\textwidth}
        \psfrag{(a)}{(a)}
	\psfrag{eps}{\footnotesize{$\varepsilon$}}
	\psfrag{P}{\footnotesize{$P$}}
	\centering
	\includegraphics[scale=0.5]{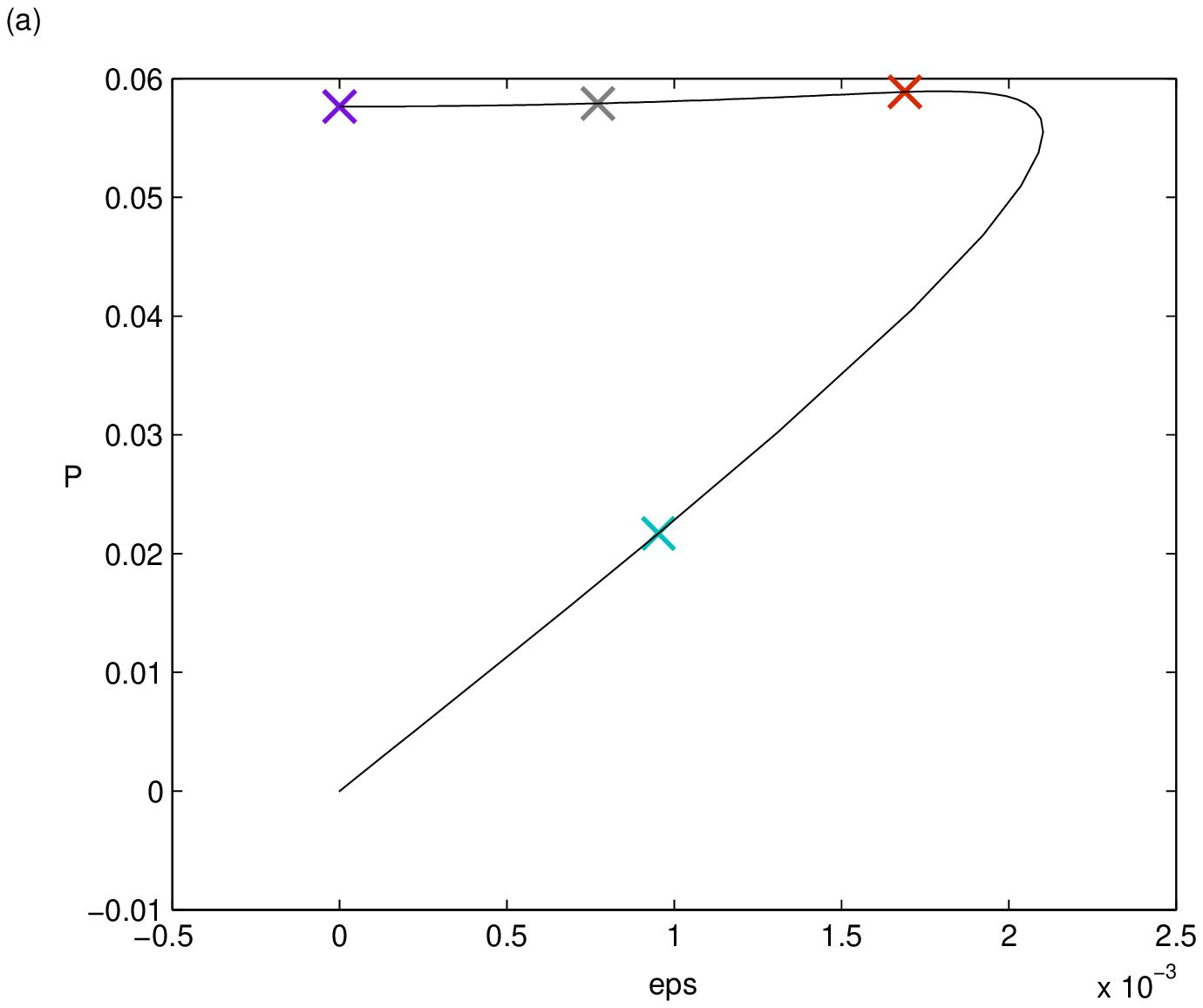}
      \end{minipage}
      \begin{minipage}{.5\textwidth}
        \psfrag{(b)}{(b)}
        \psfrag{w}{\footnotesize{$w$}}
        \psfrag{u}{\footnotesize{$u$}}
        \psfrag{z}{\footnotesize{$z$}}
	\centering
	\includegraphics[scale=0.5]{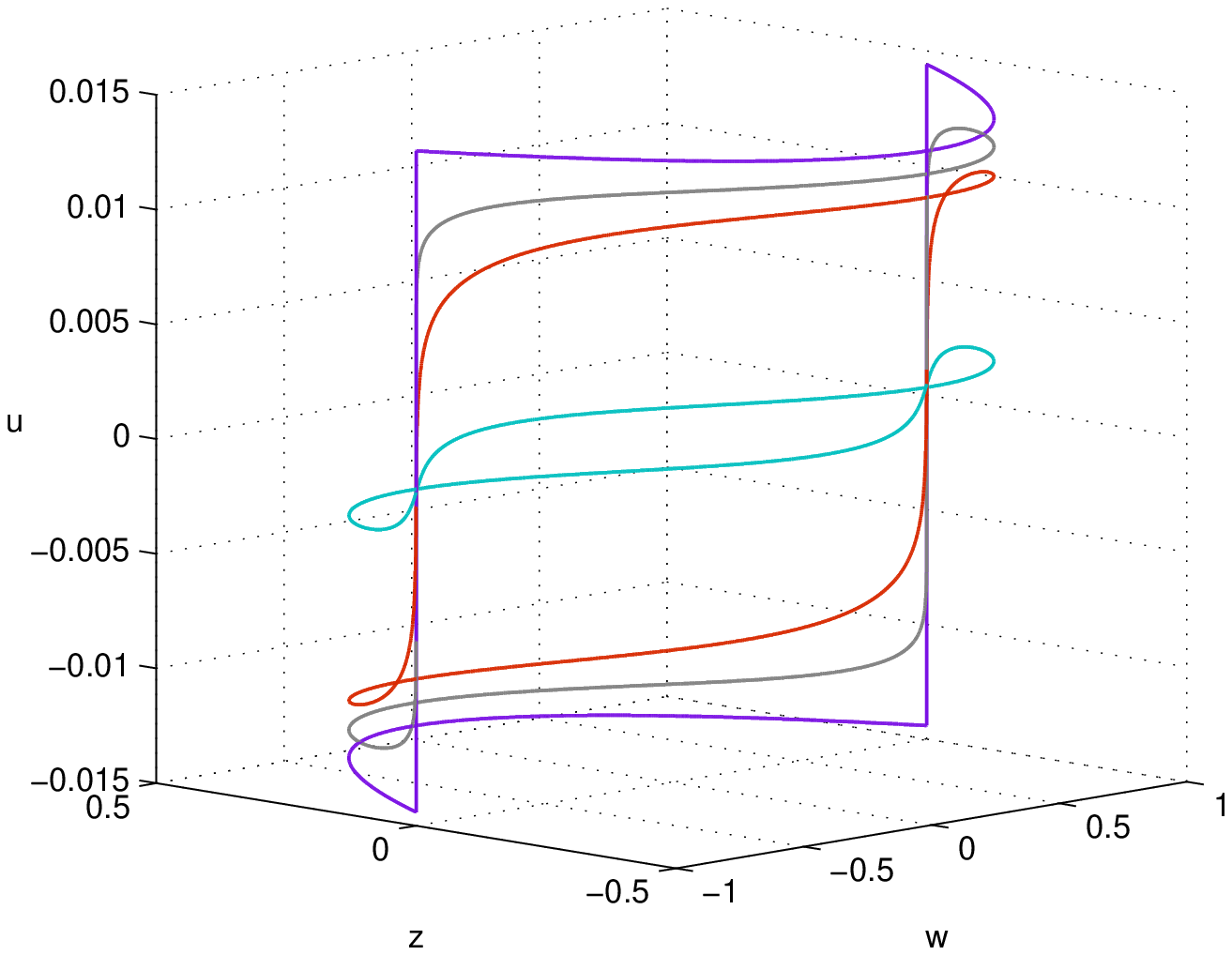}
      \end{minipage}\\
      \hspace{-1cm}
      \begin{minipage}{.5\textwidth}
        \psfrag{(b)}{(b)}
	\psfrag{epsilon}{\footnotesize{$\varepsilon$}}
	\psfrag{P}{\footnotesize{$P$}}
	\centering
	\includegraphics[scale=0.5]{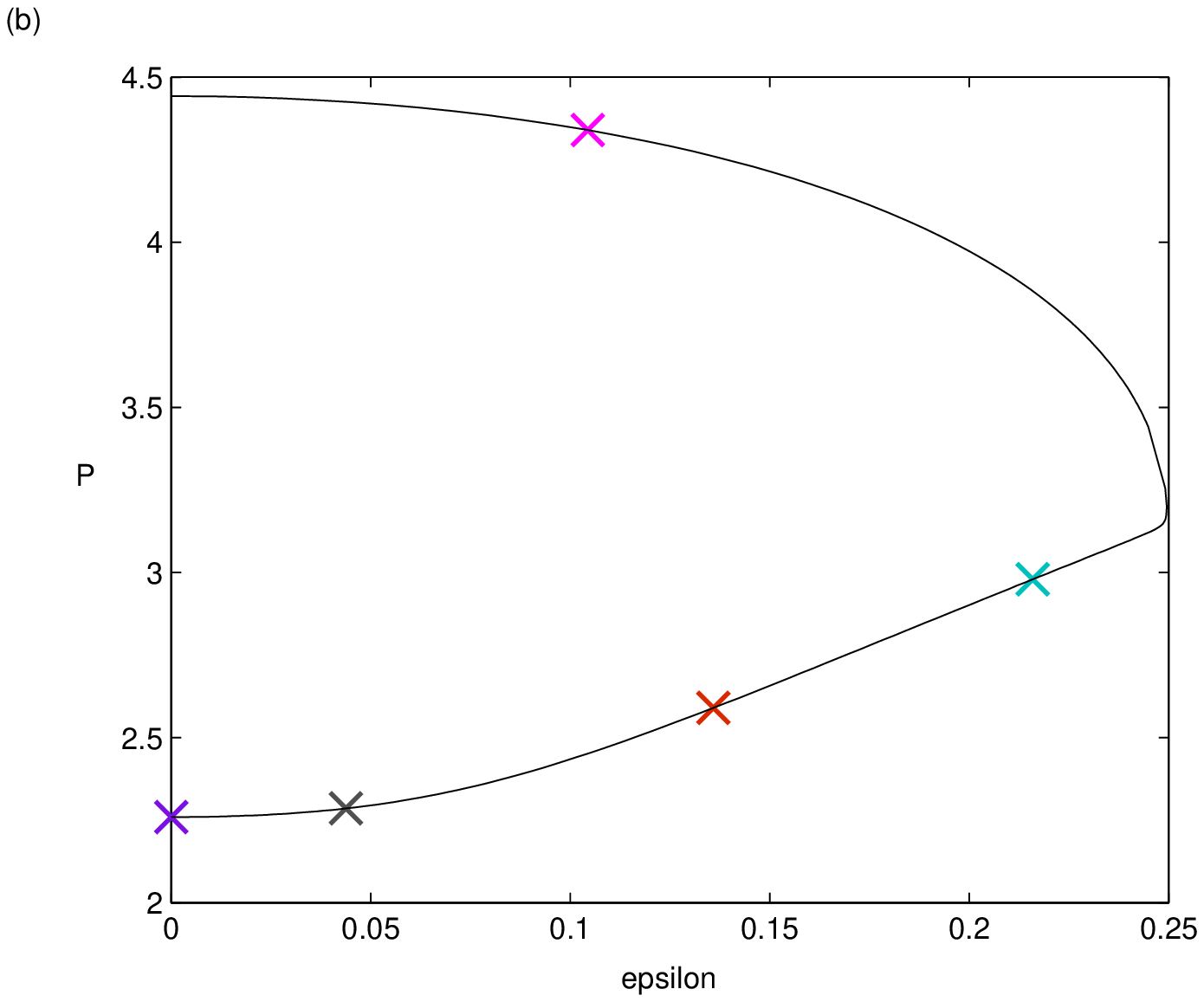}
      \end{minipage}
      \begin{minipage}{.5\textwidth}
        \psfrag{(b)}{(b)}
        \psfrag{w}{\footnotesize{$w$}}
        \psfrag{u}{\footnotesize{$u$}}
        \psfrag{z}{\footnotesize{$z$}}
	\centering
	\includegraphics[scale=0.5]{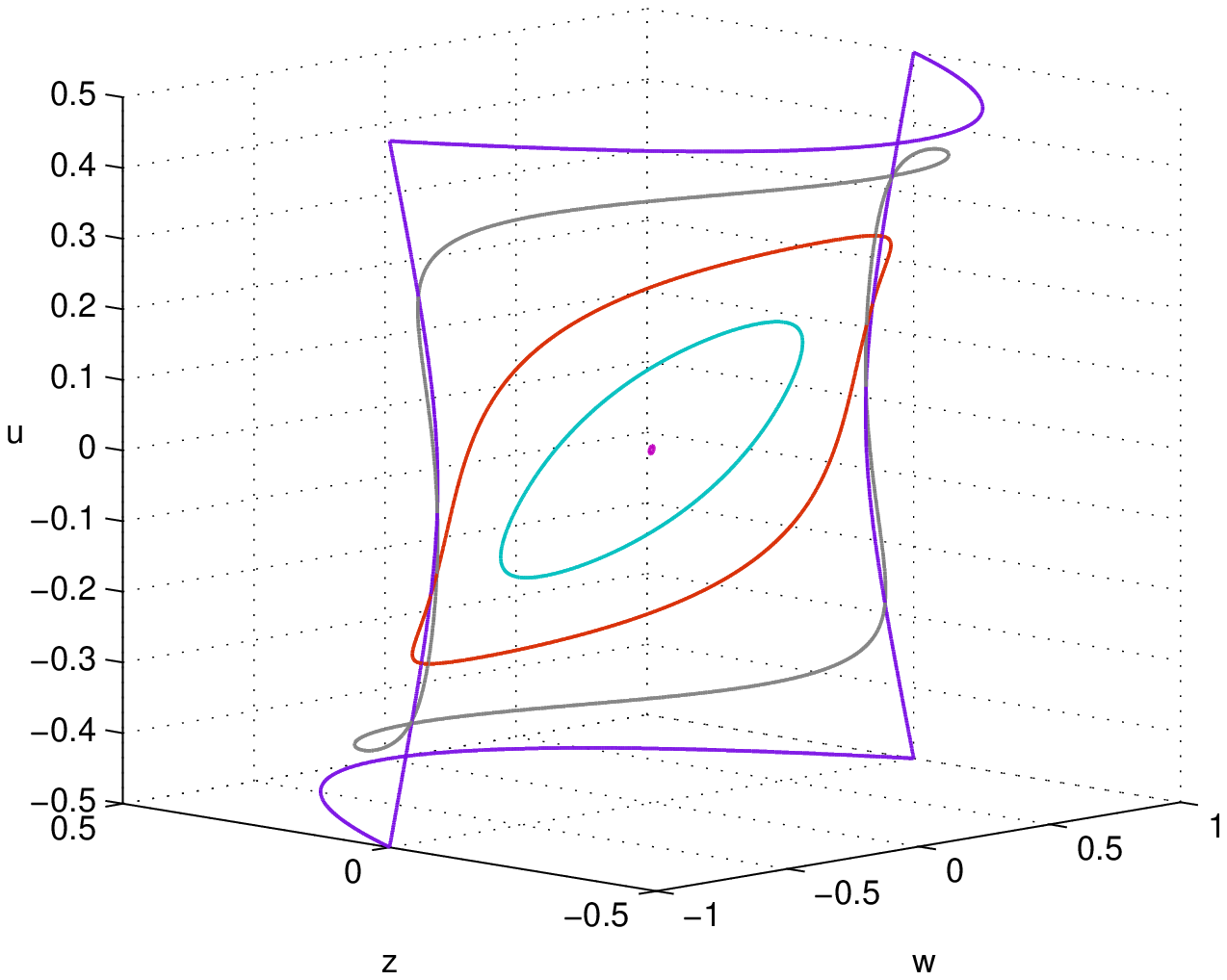}
      \end{minipage}\\
      \hspace{-1cm}
       \begin{minipage}{.5\textwidth}
        \psfrag{(c)}{(c)}
	\psfrag{epsilon}{\footnotesize{$\varepsilon$}}
	\psfrag{P}{\footnotesize{$P$}}
	\centering
	\includegraphics[scale=0.5]{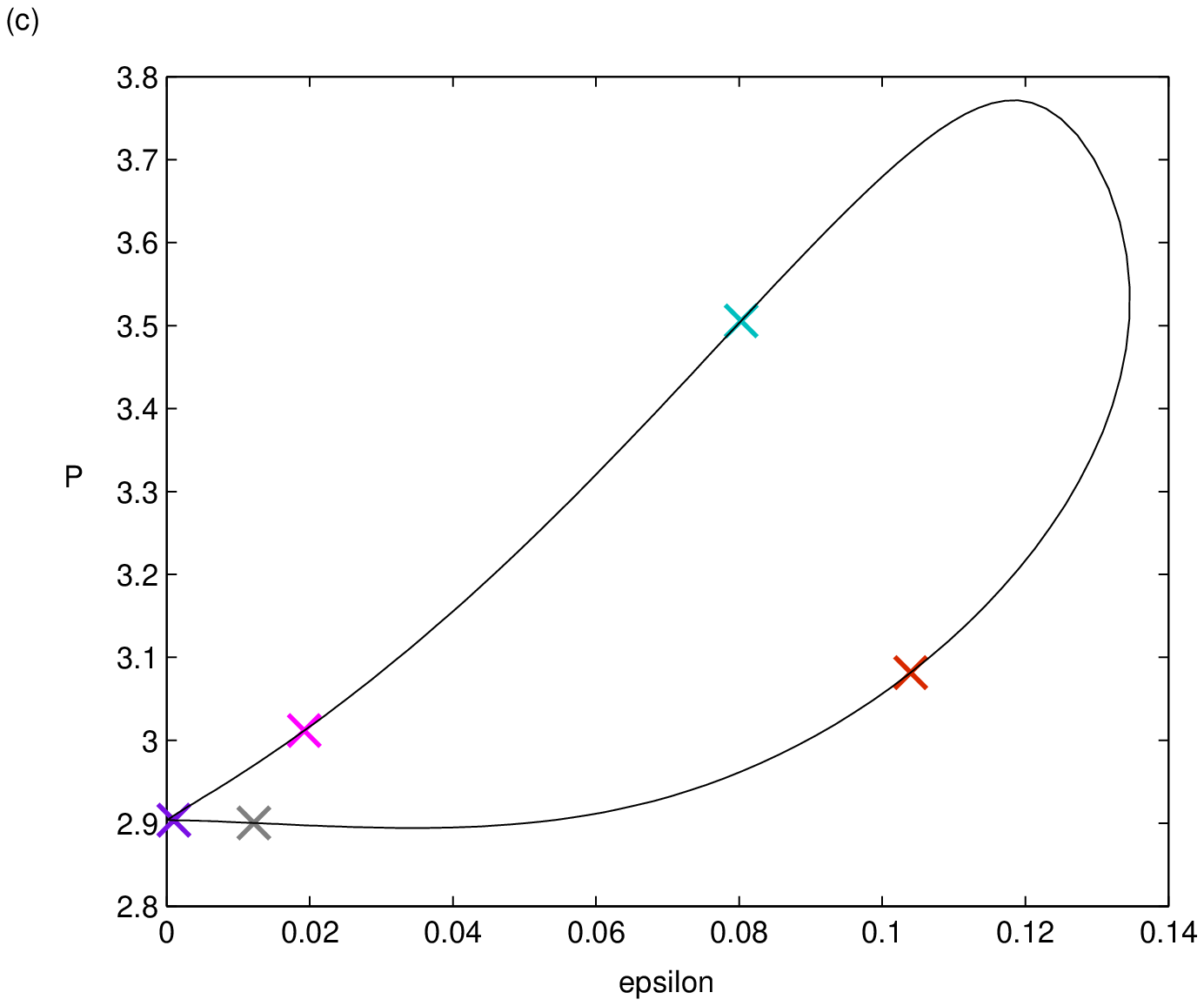}
      \end{minipage}
      \begin{minipage}{.5\textwidth}
        \psfrag{w}{\footnotesize{$w$}}
        \psfrag{u}{\footnotesize{$u$}}
        \psfrag{z}{\footnotesize{$z$}}
	\centering
	\includegraphics[scale=0.5]{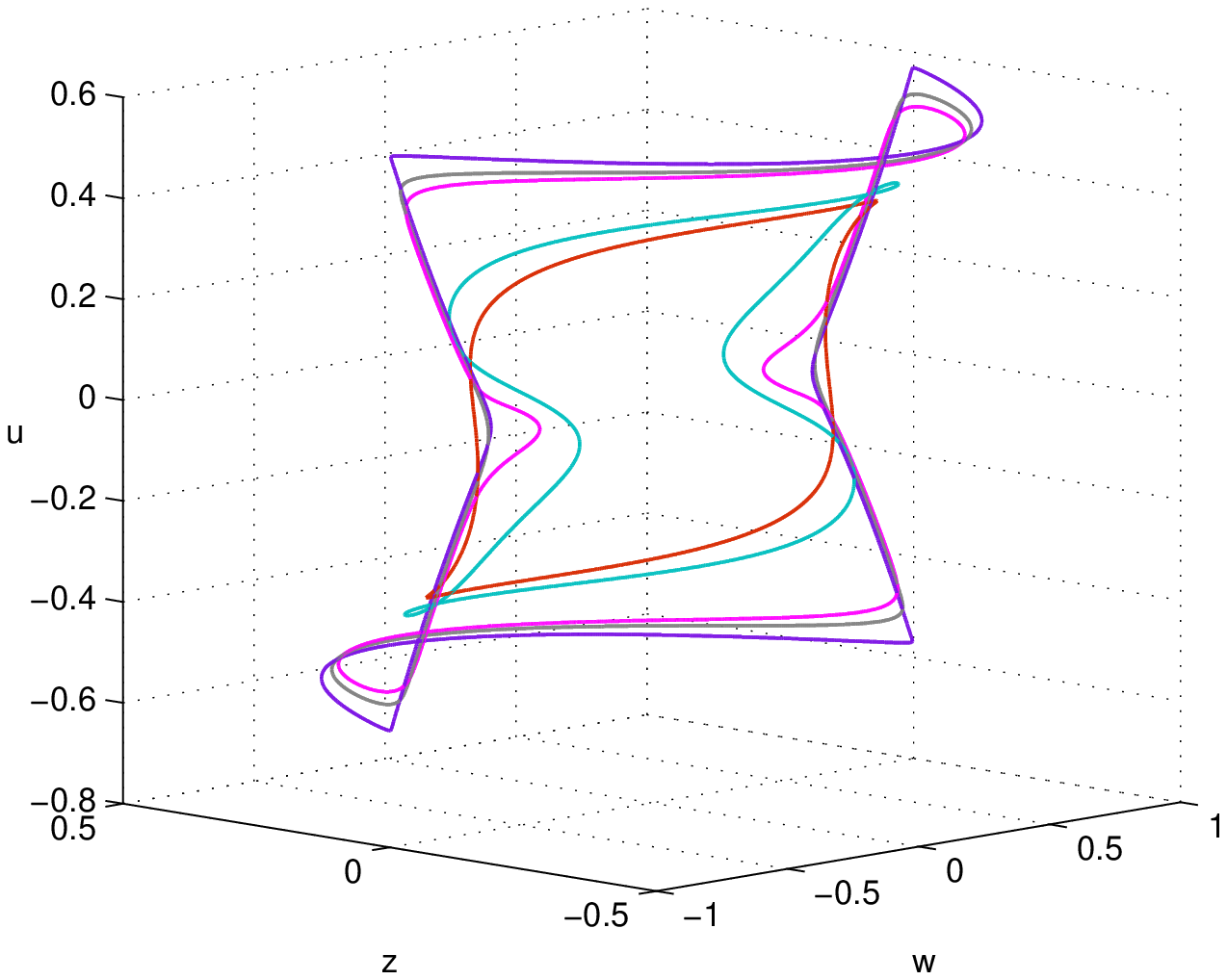}
      \end{minipage}     
      \caption{Continuation in $\varepsilon$: on the left side bifurcation diagrams 
      in $(\varepsilon, P)$ are shown, on the right the corresponding solutions in 
      $\left( w, z, u \right)$-space are displayed. (a) $\mu=\mu_l$, (b) $\mu=\mu_c$, 
      (c) $\mu=\mu_r$.}
      \label{fig:Fig4}      
      \end{figure}
      
\subsection{Period scaling}
\label{ssec:pscaling}

So far, no boundary conditions have been imposed; moreover, all the computed solutions 
are not necessarily minimizers of the functional, but only critical points. Our 
conjecture is that the interaction between the two main parameters of the system 
$\mu$ and $\varepsilon$ should allow us to obtain the true minimizers via a 
double-limit. In other words, for every value of $\varepsilon$ there is a 
corresponding orbit which minimizes the functional $\mathcal{I}^{\varepsilon}$, 
and since along this orbit the Hamiltonian has to constantly assume a certain 
value $\bar{\mu}$, the minimization process should imply a direct connection 
between the parameters. Consequently, it is interesting to investigate this 
ansatz from the numerical viewpoint.\medskip

      \begin{figure}[!ht]
      \psfrag{a}{\footnotesize{(a)}}
      \psfrag{b}{\footnotesize{(b)}}
      \psfrag{mu}{\footnotesize{$\mu$}}
      \psfrag{P}{\footnotesize{$P$}}
      \psfrag{log(eps)}{\footnotesize{$\ln(\varepsilon)$}}
      \psfrag{log(P)}{\footnotesize{$\ln(P)$}}
      \psfrag{e1}{\footnotesize{$\varepsilon=0.1$}}
      \psfrag{e2}{\footnotesize{$\varepsilon=0.01$}}
      \psfrag{e3}{\footnotesize{$\varepsilon=10^{-5}$}}
      \centering
      \includegraphics[scale=0.7]{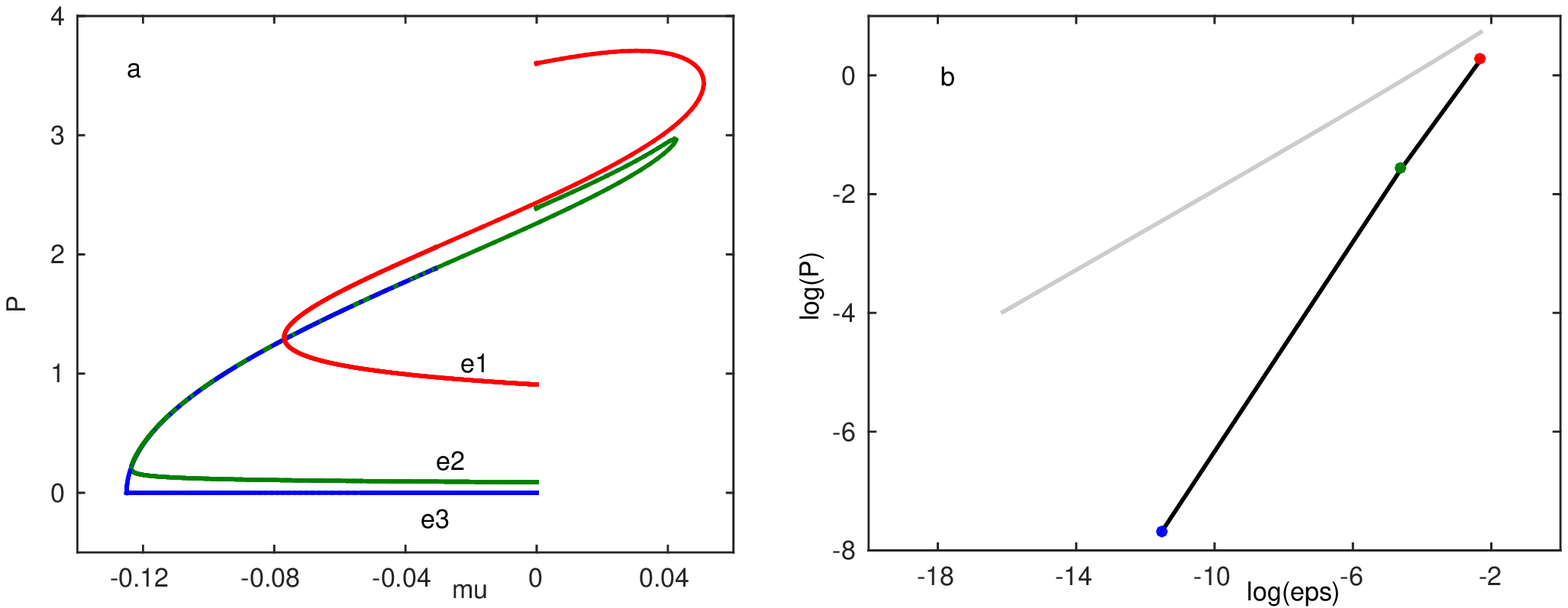}
      \caption{Illustration of two-parameter continuation. (a) Three different bifurcation
      diagrams have been computed, each starting from a solution at $\mu=0$ for three different
      values of $\varepsilon=0.1,0.01,10^{-5}$ (red, green, blue). It is already visible and
      confirmed by the computation that the sequence of leftmost fold points on each branch
      converges to $\mu=-1/8$ as $\varepsilon\rightarrow 0$. However, the period scaling law of 
      the orbits precisely at these fold points, which is shown in (b) as three dots corresponding 
      to the three folds in (a) and a suitable interpolation (black line), does not
      converge as $\cO(\varepsilon^{1/3})$ (grey reference line with slope $\frac13$).}
      \label{fig:Figtwop}
      \end{figure}
      
A first possibility is to establish a connection between the two parameters 
$\varepsilon$ and $\mu$ via a direct continuation in both parameters, starting 
from certain special points, such as the fold points detected in 
Sections~\ref{ssec:contmu}-\ref{ssec:conteps}. However, it turns out that this
process does not lead to the correct scaling law for minimizers of 
$\cI_\varepsilon$ as shown in Figure~\ref{fig:Figtwop}.\medskip

Another option is instead to check if among the critical points of the Euler-Lagrange 
equation~\eqref{eq:EL} we have numerically obtained there are also the minimizers of 
the functional $\mathcal{I}^{\varepsilon}$ respecting the power law~\eqref{eq:P_mue}. 
In~\cite{Mu}, boundary conditions on the interval $\left[ 0,1 \right]$ are also included 
in the variational formulation, and from the results obtained from the continuation in 
$\varepsilon$, one may expect that high values of $\mu$ would not be able to fit them, 
since the period is always too high. Lower values of $\mu$, instead, seem to have sufficiently 
small period. Hence, one could fix one of those (for example, $\mu_l$) and look at what happens 
as $\varepsilon \rightarrow 0$. The hope is that the $\cO(\varepsilon^{1/3})$ leading-order scaling 
for the period naturally emerges. Unfortunately, this does not happen, as we can see in 
Figure~\ref{fig:Fig5}; the lower branch seems to give a linear dependence on $\varepsilon$, 
while the upper branch gives a quadratic one.

      \begin{figure}[!ht]
      \psfrag{log(eps)}{\footnotesize{$ln(\varepsilon)$}}
      \psfrag{log(P)}{\footnotesize{$ln(P)$}}
      \centering
      \includegraphics[scale=0.5]{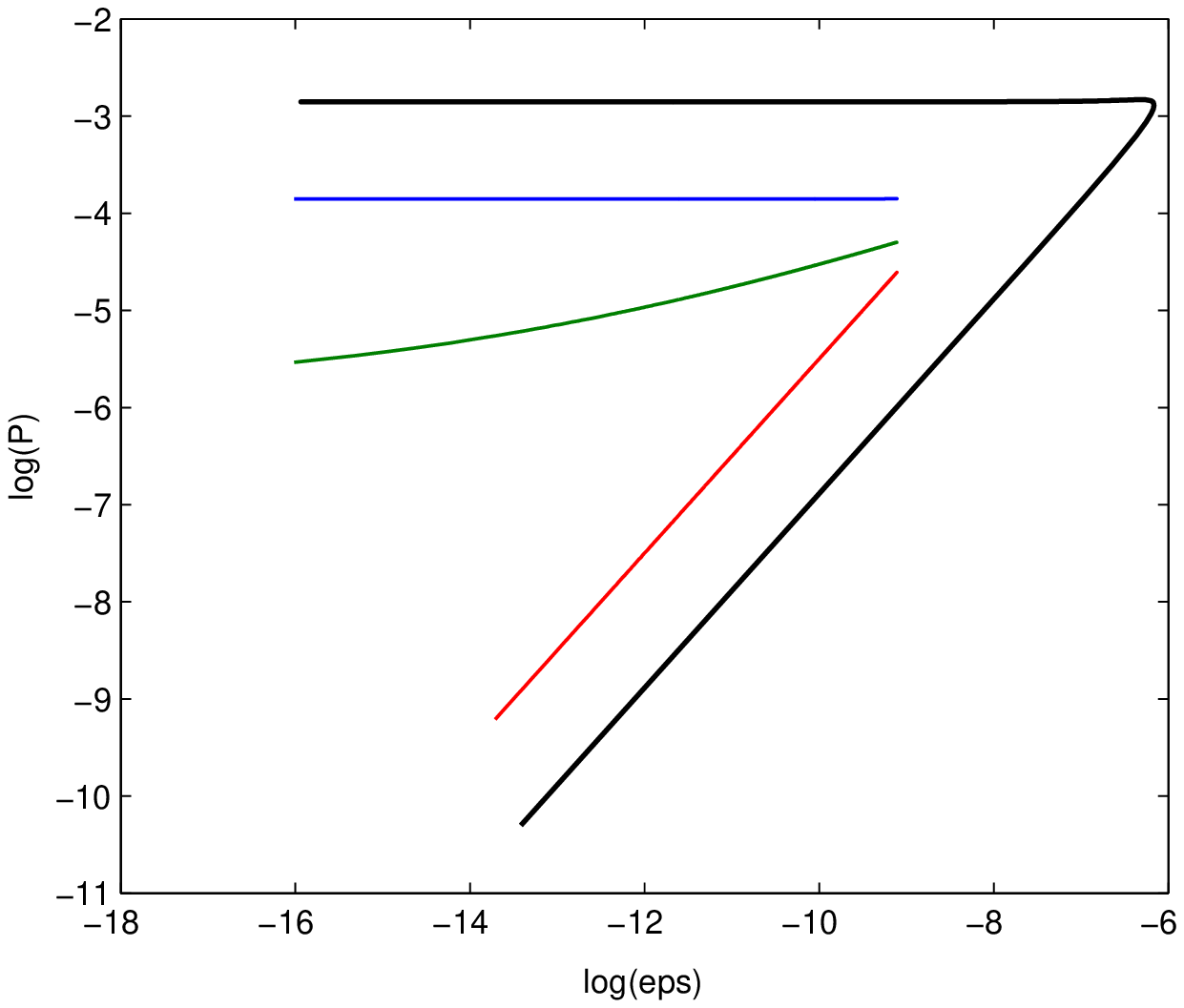}
      \caption{Possible fits of the form $P \simeq \varepsilon^{\alpha}$ for the numerical 
			data computed with $\mu=\mu_l$ (black line): $\alpha=2$, blue; $\alpha = 1/3$, 
			green; $\alpha = 1$, red.}
      \label{fig:Fig5}
      \end{figure}

The reason why from this naive approach the $\cO(\varepsilon^{1/3})$ leading-order scaling does 
not emerge lies in the lack of connection with the minimization process. However, 
Figure~\ref{fig:Fig5} demonstrates that there are several nontrivial scalings of natural 
families of periodic orbits as $\varepsilon\rightarrow 0$.\medskip

So far, we have just assumed that the Hamiltonian value of the minimizers should be ``low'', 
but indeed there is a strict connection between the values of $\varepsilon$ one is considering 
and the value of $\mu$ of the minimizers. In other words, there is not a unique value of 
$\mu$ given by the minimizers for every $\varepsilon$ small but minimizers move over different
Hamiltonian energy levels as $\varepsilon\rightarrow 0$. Starting from this consideration, 
another option, which turns out to be the correct one to recover the scaling~\eqref{eq:P_mue}, 
is to use the periodic orbits from numerical continuation to compute the numerical value of 
the functional $\mathcal{I}^{\varepsilon}$ as a function of the period $P$ fixing different 
values of $\varepsilon$ in a suitable range, such as:
\begin{equation}
\label{eq:eps_range}
 I_{\varepsilon}= \left[ 10^{-7}, 10^{-1} \right].
\end{equation}
			
Then, we obtain different parabola-shaped diagrams where we can extract the value of the period 
minimizing the functional (Figure~\ref{fig:Fig7}). When plotting these values related to the 
value of $\varepsilon$ for which they have been computed, one obtains the results shown in 
Figure~\ref{fig:Fig6}. The values numerically extracted from our solutions match the analytical 
results on the period proven by M\"uller~\eqref{eq:P_mue} when the value of $\varepsilon$ is 
sufficiently small. As $\varepsilon$ increases, the period law is less accurate, as one 
would expect. 

      \begin{figure}[H]
      \psfrag{I}{\footnotesize{$\mathcal{I}^{\bar{\varepsilon}}$}}
      \psfrag{P}{\footnotesize{$P$}}
      \centering
      \includegraphics[scale=0.5]{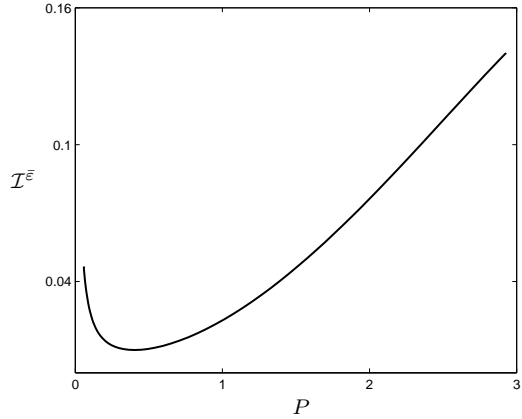}
      \caption{Parabola-shaped diagram obtained by fixing $\varepsilon=0.001$ and 
			numerically computing the value of the functional $\mathcal{I}^{\varepsilon}$ 
			along the solutions computed via continuation in \texttt{AUTO}. The plot presents 
			a minimum, and the value of $P$ corresponding to $\varepsilon$ where this 
			minimum is realized is recorded in order to check the period law~\eqref{eq:P_mue}.}
      \label{fig:Fig7}
      \end{figure}

      \begin{figure}[H]
      \begin{minipage}{.5\textwidth}
        \psfrag{(a)}{(a)}
	\psfrag{eps}{\footnotesize{$\varepsilon$}}
	\psfrag{P}{\footnotesize{$P_{min}$}}
	\centering
	\includegraphics[scale=0.55]{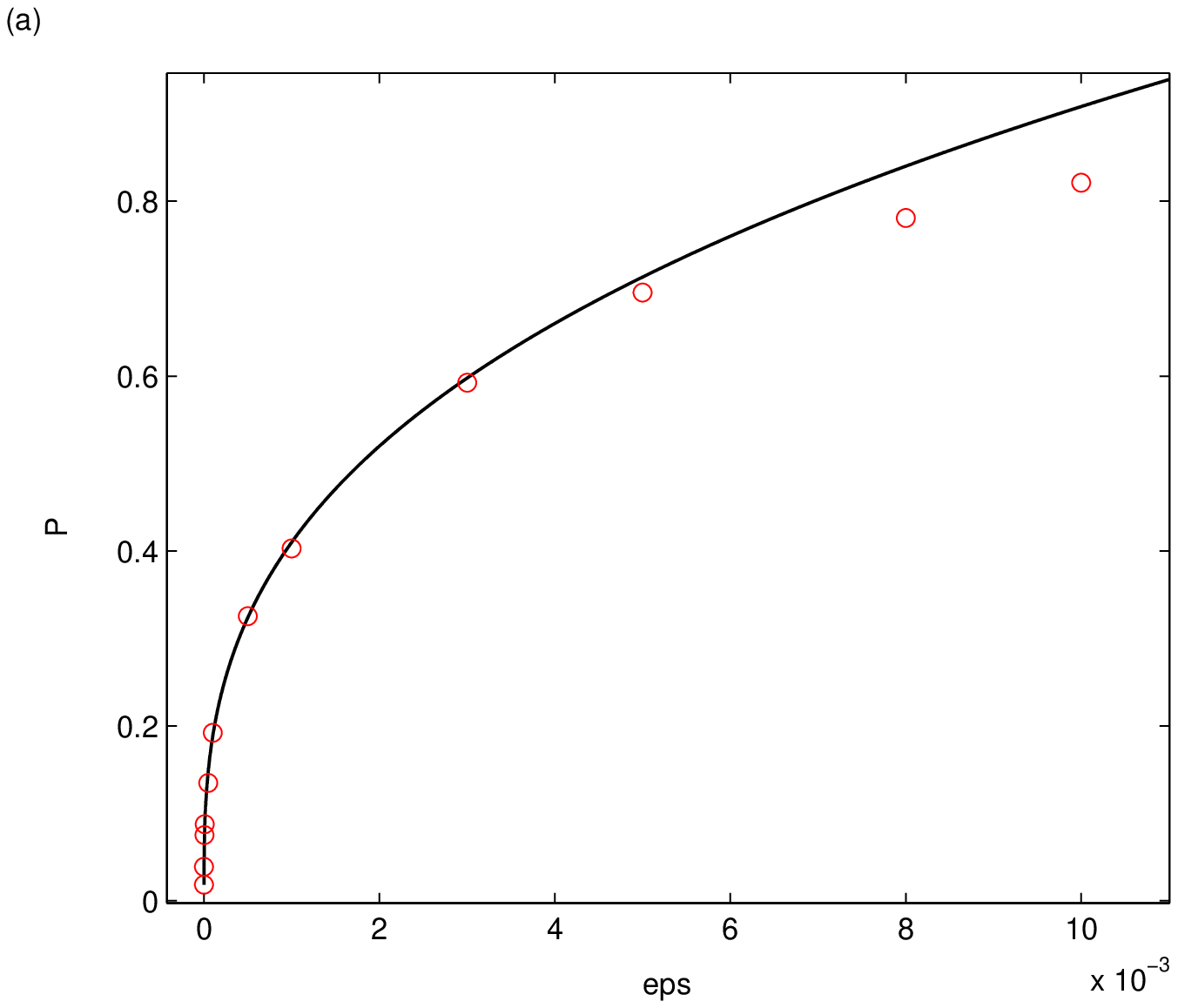}
      \end{minipage}%
      \begin{minipage}{.5\textwidth}
        \psfrag{(b)}{(b)}
	\psfrag{log(eps)}{\footnotesize{$ln(\varepsilon)$}}
	\psfrag{log(P)}{\footnotesize{$ln(P_{min})$}}
	\centering
	\includegraphics[scale=0.55]{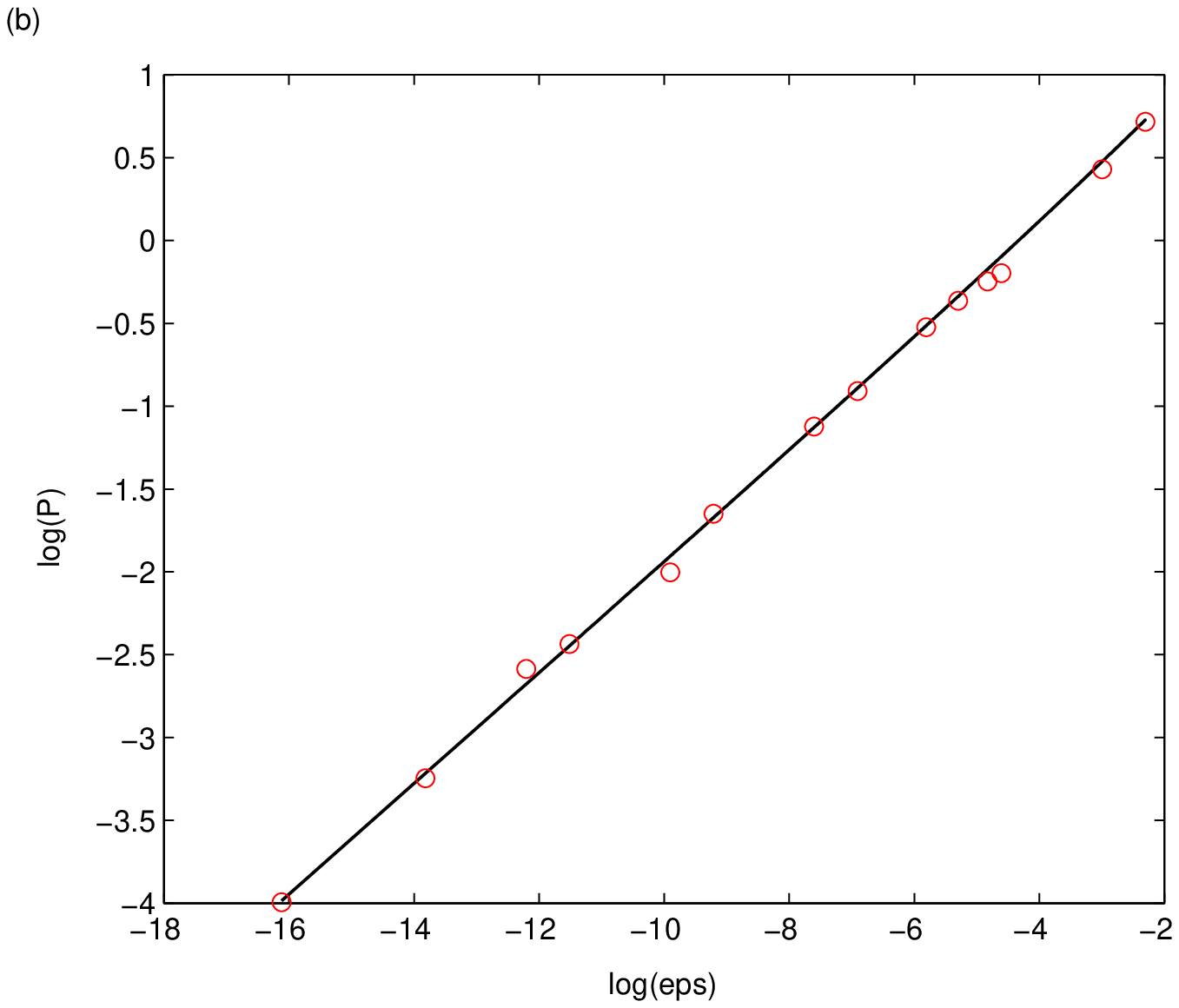}
      \end{minipage}
      \caption{Comparison between the values of $P$ minimizing $\mathcal{I}^{\varepsilon}$ 
			for several discrete values in the range $I_{\varepsilon}$ (red circles) and the period 
			law~\eqref{eq:P_mue} (black line). (a) Zoom on the range $\left[ 10^{-7}, 10^{-2} \right]$, 
			where it is expected that large values of $\varepsilon$ tend to deviate from the 
			$\cO(\varepsilon^{1/3})$ leading-order scaling, while for low values the scaling the scaling
			agrees. (b) The same plot as in (a) on a log-log scale.}
      \label{fig:Fig6}      
      \end{figure}

\section{Conclusion \& Outlook}
\label{sec:conclusion}

In summary, we have shown that geometric singular perturbation theory and numerical continuation 
methods can be very effective tools to understand nonconvex multiscale variational
problems via the Euler-Lagrange formulation. We have proven the existence of a class of singular 
periodic orbits based upon a fast-slow decomposition approach and we have shown that these orbits persist
for $\varepsilon$ small.

The geometric insight was used
to determine a starting solution for numerical continuation in the context of a reduced 
three-dimensional fast-slow system. Then we studied the dependence of periodic orbits on the 
singular perturbation parameter as well as the Hamiltonian energy level set parameter arising
in the reduction from a four- to a three-dimensional system. The parameter space is structured
by several fold points. Furthermore, we were able to study the shape of non-minimizing periodic
orbits for very broad classes of parameters. Finally, we showed that several natural scaling 
laws for non-minimizing sequences of periodic solutions exist and also confirmed numerically
the leading-order scaling predicted by M\"uller for minimizing sequences.\medskip

Based upon this work, there are several open problems as well as generalizations one might
consider. In particular, it would be desirable to extend the persistence result 
to the general class of singularly-perturbed Hamiltonian fast-slow systems~\eqref{spHs}; this
is the subject of ongoing work.\medskip

Another important observation of our numerical study are the intricate orbits that seem to
arise when parts of the slow segments start to interact with the singularities $\cL_\pm$ where
the critical manifold is not normally hyperbolic. The natural conjecture is that the 
additional small fast loops that we observe numerically could correspond to homoclinic ``excursions''
in the fast subsystem anchored at points close to $\cL_\pm$. The blow-up method~\cite{DumortierRoussarie} is 
likely to provide an excellent tool to resolve the non-normally hyperbolic singularities; see 
{e.g.}~\cite{GucwaSzmolyan2,KuehnSzmolyan} where the existence of complicated fast-slow 
periodic orbits involving loss of normal hyperbolicity is proven.\medskip 

The construction of an initial orbit has been one of the hardest problems to tackle. It was solved 
using analytical and numerical tools, after discarding several other plausible approaches. 
The SMST algorithm \cite{SMST} has been a helpful tool to determine good starting points on the slow 
manifolds and then use an initial value solver to obtain segments of a complete whole orbit. Although
our approach works well in practical computations, there are interesting deep numerical analysis
questions still to be answered regarding the interplay between certain classes of fast-slow ``initial
guess'' starting orbits and the success or failure of Newton-type methods for the associated BVPs. In
particular, can one prove certain geometric conditions or restrictions on $\varepsilon$ to guarantee
the convergence for the first solution?\medskip

Another highly relevant direction would be to extend our approach to more general classes of 
functionals. There are many different singularly-perturbed variational problems, arising 
e.g.~in materials science, to which one may apply the techniques presented here. In this context,
it is important to emphasize that we expect that particularly other non-convex functionals could 
be excellent candidates for future work.\medskip

From the viewpoint of applications, it would be interesting to study the practical relevance of 
non-minimizing sequences of periodic solutions. Although we expect the long-term behavior to be 
governed by minimizers, it is evident that non-minimizing periodic orbits can have a high impact 
on time-dependent dynamics, {e.g.}, either via transient behavior, via noise-induced phase 
transitions, or as dynamical boundaries between different regimes.\medskip

\textbf{Acknowledgements:} AI and PS would like to thank the Fonds zur F\"orderung der 
wissenschaftlichen Forschung (FWF) for support via a doctoral school (project W1245). 
CK~would like to thank the VolkswagenStiftung for
support via a Lichtenberg professorship. CK and PS also acknowledge partial support 
of the European Commission (EC/REA) via a Marie-Curie International Reintegration Grant (MC-IRG). 

\bibliographystyle{plain}
{\footnotesize
\bibliography{references}}

\begin{thebibliography}{10}

\bibitem{Abe01}
R.~Abeyaratne, K.~Bhattacharya, and J.K. Knowles.
\newblock Strain-energy functions with multiple local minima: modeling phase
  transformations using finite thermoelasticity.
\newblock In {\em Nonlinear elasticity: theory and applications}, volume 283 of
  {\em London Math. Soc. Lecture Note Ser.}, pages 433--490. Cambridge Univ.
  Press, Cambridge, 2001.

\bibitem{Abe96B}
R.~Abeyaratne and J.K. Knowles.
\newblock {\em Evolution of phase transitions: a continuum theory}.
\newblock Cambridge University Press, 2006.

\bibitem{Bal04}
J.M. Ball.
\newblock Mathematical models of martensitic microstructure.
\newblock {\em Materials Science and Engineering: A}, 378(1):61--69, 2004.

\bibitem{Bal87}
J.M. Ball and R.D. James.
\newblock Fine phase mixtures as minimizers of energy.
\newblock {\em Arch. Rational Mech. Anal.}, 100(1):13--52, 1987.

\bibitem{Bha93}
K.~Bhattacharya.
\newblock Comparison of the geometrically nonlinear and linear theories of
  martensitic transformation.
\newblock {\em Continuum Mechanics and Thermodynamics}, 5(3):205--242, 1993.

\bibitem{Bha03B}
K.~Bhattacharya.
\newblock {\em Microstructure of martensite: why it forms and how it gives rise
  to the shape-memory effect}, volume~2.
\newblock Oxford University Press, 2003.

\bibitem{CarterSandstede}
P.~Carter and B.~Sandstede.
\newblock Fast pulses with oscillatory tails in the {FitzHugh-Nagumo} system.
\newblock {\em SIAM J. Math. Anal.}, 47(5):3393--3441, 2015.

\bibitem{Sneydetal}
A.R. Champneys, V.~Kirk, E.~Knobloch, B.E. Oldeman, and J.~Sneyd.
\newblock {When Shil'nikov meets Hopf in excitable systems}.
\newblock {\em SIAM J. Appl. Dyn. Syst.}, 6(4):663--693, 2007.

\bibitem{Dac15}
B.~Dacorogna.
\newblock {\em Introduction to the calculus of variations}.
\newblock Imperial College Press, London, third edition, 2015.

\bibitem{DesrochesKaperKrupa}
M.~Desroches, T.J. Kaper, and M.~Krupa.
\newblock Mixed-mode bursting oscillations: dynamics created by a slow passage
  through spike-adding canard explosion in a square-wave burster.
\newblock {\em Chaos}, 23:046106, 2013.

\bibitem{DesrochesKrauskopfOsinga2}
M.~Desroches, B.~Krauskopf, and H.M. Osinga.
\newblock Numerical continuation of canard orbits in slow-fast dynamical
  systems.
\newblock {\em Nonlinearity}, 23(3):739--765, 2010.

\bibitem{AU}
E.~Doedel.
\newblock A{UTO}: a program for the automatic bifurcation analysis of
  autonomous systems.
\newblock In {\em Proceedings of the {T}enth {M}anitoba {C}onference on
  {N}umerical {M}athematics and {C}omputing, {V}ol. {I} ({W}innipeg, {M}an.,
  1980)}, volume~30, pages 265--284, 1981.

\bibitem{Dol03B}
G.~Dolzmann.
\newblock {\em Variational methods for crystalline microstructure-analysis and
  computation}, volume 1803.
\newblock Springer Science \& Business Media, 2003.

\bibitem{DumortierRoussarie}
F.~Dumortier and R.~Roussarie.
\newblock {\em Canard Cycles and Center Manifolds}, volume 121 of {\em Memoirs
  Amer. Math. Soc.}
\newblock AMS, 1996.

\bibitem{Feni}
N.~Fenichel.
\newblock Geometric singular perturbation theory for ordinary differential
  equations.
\newblock {\em J. Differential Equations}, 31(1):53--98, 1979.

\bibitem{FitzHugh}
R.~FitzHugh.
\newblock Mathematical models of threshold phenomena in the nerve membrane.
\newblock {\em Bull. Math. Biophysics}, 17:257--269, 1955.

\bibitem{Gel02}
V.~Gelfreich and L.~Lerman.
\newblock Almost invariant elliptic manifold in a singularly perturbed
  {H}amiltonian system.
\newblock {\em Nonlinearity}, 15(2):447--457, 2002.

\bibitem{Giu12}
A.~Giuliani and S.~M{\"u}ller.
\newblock Striped periodic minimizers of a two-dimensional model for
  martensitic phase transitions.
\newblock {\em Comm. Math. Phys.}, 309(2):313--339, 2012.

\bibitem{GrLo}
M~Grinfeld and GJ~Lord.
\newblock Bifurcations in the regularized ericksen bar model.
\newblock {\em Journal of Elasticity}, 90(2):161--173, 2008.

\bibitem{SMST}
J.~Guckenheimer and C.~Kuehn.
\newblock Computing slow manifolds of saddle type.
\newblock {\em SIAM J. Appl. Dyn. Syst.}, 8(3):854--879, 2009.

\bibitem{GuckenheimerKuehn1}
J.~Guckenheimer and C.~Kuehn.
\newblock {Homoclinic orbits of the FitzHugh-Nagumo equation: The singular
  limit}.
\newblock {\em DCDS-S}, 2(4):851--872, 2009.

\bibitem{GuckenheimerKuehn3}
J.~Guckenheimer and C.~Kuehn.
\newblock {Homoclinic orbits of the FitzHugh-Nagumo equation: Bifurcations in
  the full system}.
\newblock {\em SIAM J. Appl. Dyn. Syst.}, 9:138--153, 2010.

\bibitem{GuckenheimerLaMar}
J.~Guckenheimer and D.~LaMar.
\newblock Periodic orbit continuation in multiple time scale systems.
\newblock In {\em Understanding Complex Systems: Numerical continuation methods
  for dynamical systems}, pages 253--267. Springer, 2007.

\bibitem{GuckenheimerMeerkamp}
J.~Guckenheimer and P.~Meerkamp.
\newblock Bifurcation analysis of singular {Hopf} bifurcation in
  {$\mathbb{R}^3$}.
\newblock {\em SIAM J. Appl. Dyn. Syst.}, 11(4):1325--1359, 2012.

\bibitem{Hea}
Timothy~J Healey and Ulrich Miller.
\newblock Two-phase equilibria in the anti-plane shear of an elastic solid with
  interfacial effects via global bifurcation.
\newblock In {\em Proceedings of the Royal Society of London A: Mathematical,
  Physical and Engineering Sciences}, volume 463, pages 1117--1134. The Royal
  Society, 2007.

\bibitem{GSPT}
C.K.R.T. Jones.
\newblock Geometric singular perturbation theory.
\newblock In {\em Dynamical systems ({M}ontecatini {T}erme, 1994)}, volume 1609
  of {\em Lecture Notes in Math.}, pages 44--118. Springer, Berlin, 1995.

\bibitem{JonesKopellLanger}
C.K.R.T. Jones, N.~Kopell, and R.~Langer.
\newblock Construction of the {FitzHugh-Nagumo} pulse using differential forms.
\newblock In {\em Multiple-Time-Scale Dynamical Systems}, pages 101--113.
  Springer, 2001.

\bibitem{Kha83B}
A.G. Khachaturyan.
\newblock {\em Theory of structural transformations in solids}.
\newblock Courier Corporation, 2013.

\bibitem{Kha69}
A.G. Khachaturyan and G.~Shatalov.
\newblock Theory of macroscopic periodicity for a phase transition in the solid
  state.
\newblock {\em Soviet Phys. JETP}, 29(3):557--561, 1969.

\bibitem{Koh94}
R.V. Kohn and S.~M{\"u}ller.
\newblock Surface energy and microstructure in coherent phase transitions.
\newblock {\em Comm. Pure Appl. Math.}, 47(4):405--435, 1994.

\bibitem{GucwaSzmolyan2}
I.~Kosiuk and P.~Szmolyan.
\newblock Scaling in singular perturbation problems: blowing-up a relaxation
  oscillator.
\newblock {\em SIAM J. Appl. Dyn. Syst.}, 10(4):1307--1343, 2011.

\bibitem{Kristiansen}
K.U. Kristiansen.
\newblock Computation of saddle-type slow manifolds using iterative methods.
\newblock {\em SIAM J. Appl. Dyn. Syst.}, 14(2):1189--1227, 2015.

\bibitem{KrupaSandstedeSzmolyan}
M.~Krupa, B.~Sandstede, and P.~Szmolyan.
\newblock Fast and slow waves in the {FitzHugh-Nagumo} equation.
\newblock {\em J. Differential Equat.}, 133:49--97, 1997.

\bibitem{Kue}
C.~Kuehn.
\newblock {\em Multiple time scale dynamics}, volume 191 of {\em Applied
  Mathematical Sciences}.
\newblock Springer, Cham, 2015.

\bibitem{KuehnSzmolyan}
C.~Kuehn and P.~Szmolyan.
\newblock Multiscale geometry of the {Olsen} model and non-classical relaxation
  oscillations.
\newblock {\em J. Nonlinear Sci.}, 25(3):583--629, 2015.

\bibitem{Mu}
S.~M{\"u}ller.
\newblock Singular perturbations as a selection criterion for periodic
  minimizing sequences.
\newblock {\em Calc. Var. Partial Differential Equations}, 1(2):169--204, 1993.

\bibitem{Mu_rev}
S.~M{\"u}ller.
\newblock Variational models for microstructure and phase transitions.
\newblock In {\em Calculus of variations and geometric evolution problems
  ({C}etraro, 1996)}, volume 1713 of {\em Lecture Notes in Math.}, pages
  85--210. Springer, Berlin, 1999.

\bibitem{NAG}
J.~Nagumo, S.~Arimoto, and S.~Yoshizawa.
\newblock An active pulse transmission line simulating nerve axon.
\newblock {\em Proceedings of the IRE}, 50(10):2061--2070, 1962.

\bibitem{Ped00B}
P.~Pedregal.
\newblock {\em Variational methods in nonlinear elasticity}.
\newblock Society for Industrial and Applied Mathematics (SIAM), Philadelphia,
  PA, 2000.

\bibitem{Pit03B}
M.~Pitteri and G.~Zanzotto.
\newblock {\em Continuum models for phase transitions and twinning in
  crystals}, volume~19 of {\em Applied Mathematics (Boca Raton)}.
\newblock Chapman \& Hall/CRC, Boca Raton, FL, 2003.

\bibitem{Roi78}
A.L. Roitburd.
\newblock Martensitic transformation as a typical phase transformation in
  solids.
\newblock {\em Solid state physics}, 33:317--390, 1978.

\bibitem{ST}
C.~Soto-Trevi{\~n}o.
\newblock A geometric method for periodic orbits in singularly-perturbed
  systems.
\newblock In {\em Multiple-time-scale dynamical systems ({M}inneapolis, {MN},
  1997)}, volume 122 of {\em IMA Vol. Math. Appl.}, pages 141--202. Springer,
  New York, 2001.

\bibitem{Tru95}
L.~Truskinovsky and G.~Zanzotto.
\newblock Finite-scale microstructures and metastability in one-dimensional
  elasticity.
\newblock {\em Meccanica}, 30(5):577--589, 1995.

\bibitem{Tru96}
L.~Truskinovsky and G.~Zanzotto.
\newblock Ericksen's bar revisited: Energy wiggles.
\newblock {\em Journal of the Mechanics and Physics of Solids},
  44(8):1371--1408, 1996.

\bibitem{Tsaneva-AtanasovaOsingaRiessSherman}
K.T. Tsaneva-Atanasova, H.M. Osinga, T.~Riess, and A.~Sherman.
\newblock Full system bifurcation analysis of endocrine bursting models.
\newblock {\em J. Theor. Biol.}, 264(4):1133--1146, 2010.

\bibitem{Vain1}
A~Vainchtein, T~Healey, P~Rosakis, and L~Truskinovsky.
\newblock The role of the spinodal region in one-dimensional martensitic phase
  transitions.
\newblock {\em Physica D: Nonlinear Phenomena}, 115(1):29--48, 1998.

\bibitem{Yip}
Nung~Kwan Yip.
\newblock Structure of stable solutions of a one-dimensional variational
  problem.
\newblock {\em ESAIM: Control, Optimisation and Calculus of Variations},
  12(4):721--751, 2006.

\end{thebibliography}

\end{document}